\documentclass[11pt]{amsart}
\usepackage{geometry}                
\usepackage{graphicx}
\usepackage{amssymb}
\usepackage{epstopdf}
\usepackage{cancel}

\usepackage{amsmath,amscd}

\usepackage{xy}                                       

\usepackage{rotating}

\usepackage{multirow, array}

\usepackage{tikz}
\usetikzlibrary{matrix,arrows}

\usepackage{mathtools}

\allowdisplaybreaks 

\title{Explicit formulas for Masses of Ternary Quadratic Lattices of varying determinant over Number fields}
\author{Jonathan Hanke}
\date{\today}                                           

\begin{document}
\maketitle

\begin{abstract}
This paper gives explicit formulas for the formal 
total mass Dirichlet series for integer-valued ternary quadratic lattices of varying determinant and fixed signature over number fields $F$ where $p=2$ splits completely.   
We prove this by using local genus invariants and local mass formulas to compute the local factors of the theory developed in  \cite{Ha-masses_of_varying_det}. 
When the signature is positive definite these formulas be checked against tables of positive definite ternary quadratic forms over $\mathbb{Z}$, and we have written specialized software \cite{Ha-Sage-Dirichlet, Ha-Sage-Tables, Ha-Sage-Mass} which checks these results when 
the Hessian determinant is $\leq 2 \times 10^4$.
This work can also be applied to study the 2-parts of class groups of cubic fields (e.g. see \cite{Bh-Ha-Sh}).  
\end{abstract}

\newtheorem{thm}{Theorem}[section]  
\newtheorem{lem}[thm]{Lemma}  
\newtheorem{cor}[thm]{Corollary}  
\newtheorem{defn}[thm]{Definition}  
\newtheorem{rem}[thm]{Remark}  
\newtheorem{conj}[thm]{Conjecture}
\newtheorem{question}[thm]{Question}

\renewcommand{\O}{\mathrm{O}}
\newcommand{\SO}{\mathrm{SO}}
\newcommand{\SL}{\mathrm{SL}}
\newcommand{\GL}{\mathrm{GL}}
\newcommand{\ord}{\mathrm{ord}}
\newcommand{\Gen}{\mathrm{Gen}}
\newcommand{\Sym}{\mathrm{Sym}}
\newcommand{\sgn}{\mathrm{sgn}}
\newcommand{\SqCl}{\mathrm{SqCl}}
\newcommand{\Mass}{\mathrm{Mass}}
\newcommand{\Aut}{\mathrm{Aut}}

\newcommand{\N}{\mathbb{N}}
\newcommand{\Z}{\mathbb{Z}}
\newcommand{\Q}{\mathbb{Q}}
\newcommand{\R}{\mathbb{R}}
\newcommand{\C}{\mathbb{C}}
\newcommand{\F}{\mathbb{F}}
\renewcommand{\S}{\mathbb{S}}
\newcommand{\T}{\mathbb{T}}
\newcommand{\U}{\mathbb{U}}
\renewcommand{\O}{\mathbb{O}}
\newcommand{\E}{\mathbb{E}}
\newcommand{\A}{\mathbb{A}}

\newcommand{\x}{\vec x}
\newcommand{\y}{\vec y}
\renewcommand{\c}{\vec c}

\newcommand{\M}{\mathcal{M}}
\renewcommand{\H}{\mathcal{H}}

\renewcommand{\[}{\left[}
\renewcommand{\]}{\right]}
\renewcommand{\(}{\left(}
\renewcommand{\)}{\right)}

\newcommand{\ra}{\rightarrow}
\newcommand{\al}{\alpha}
\newcommand{\ve}{\varepsilon}

\newcommand{\leg}[2]{\(\frac{#1}{#2}\)}

\newcommand{\p}{\mathfrak{p}}
\newcommand{\q}{\mathfrak{q}}
\newcommand{\OF}{\mathcal{O}_F}
\newcommand{\Op}{\mathcal{O}_\p}
\newcommand{\Oq}{\mathcal{O}_\q}
\newcommand{\Ov}{\mathcal{O}_v}

\newcommand{\Supp}{\mathrm{Supp}}

\newcommand{\n}{\mathfrak{n}}
\newcommand{\s}{\mathfrak{s}}
\renewcommand{\v}{\mathfrak{v}}
\newcommand{\g}{\mathcal{G}}

\newcommand{\I}{\mathfrak{I}}

\renewcommand{\P}{\mathcal{P}}
\newcommand{\G}{\mathcal{G}}

\newcommand{\SqAfInt}{\SqCl(\A_{F, \mathbf{f}}^\times, U_\mathbf{f})}
\newcommand{\SqAf}{\SqCl(\A_{F, \mathbf{f}}^\times)}

\newcommand{\SqFInt}{\SqCl(F^\times, \OF^\times)}
\newcommand{\SqF}{\SqCl(F^\times)}

\newcommand{\SqFpInt}{\SqCl(F_\p^\times, \Op^\times)}
\newcommand{\SqFp}{\SqCl(F_\p^\times)}

\newcommand{\SqFvInt}{\SqCl(F_v^\times, \Ov^\times)}
\newcommand{\SqFv}{\SqCl(F_v^\times)}

\newcommand{\pip}{\pi_\p}

\newcommand{\TableHasseOne}
{
\begin{sideways}
\begin{minipage}{\textheight}
$$
\begin{tabular}{  c ||  c  |  c  ||  c   | c  || c  | c | c |}
\multirow{2}{*}{\text{\#}}   & Allowed  & Partial Local & \# of &  Trains in Valuation & \multirow{2}{*}{Cases }& \multicolumn{2}{c|}{\# of Hasse Invariants}
\\
& $\nu = \ord_2(\det_G(Q))$ & Genus Symbol & Trains & Adjustment &   & $\hspace{.13in}c_2 = 1\hspace{.13in}$ & $c_2 = -1$
\\
\hline 
\hline
%
%
%
1 & $\nu = 0$ & (3) & 1 & 0 & -- & 1 & 1 
\\
\hline
\hline
2 & $\nu = 1$ & (2, 1) & 1 & 1 & -- & 1 & 1 
\\
\hline
\multirow{2}{*}{3} & \multirow{2}{*}{$\nu = 2$} & \multirow{2}{*}{(2; 1)} & \multirow{2}{*}{1} &  \multirow{2}{*}{0} & $u \equiv 1_{(4)}$ 
    & 1 & 2 
    \\
 & & &  & & $u \equiv 3_{(4)}$ 
    & 2 & 1 
    \\
\hline
\multirow{4}{*}{4} & \multirow{4}{*}{$\nu = b \geq 3$} & \multirow{4}{*}{(2 :: 1)} & \multirow{4}{*}{2} & \multirow{2}{*}{0 if $b$ is even} & $u \equiv 1_{(4)}$
    & 2 & 4 
    \\
 &  & &  &  & $u \equiv 3_{(4)}$
    & 4 & 2 
    \\
\cline{5-8}
 & &  & & \multirow{2}{*}{1 if $b$ is odd} & \multirow{2}{*}{--} 
    & \multirow{2}{*}{3} & \multirow{2}{*}{3}  
    \\
 & &  & &  & 
    &  &   
    \\
\hline
\multirow{2}{*}{5} & \multirow{2}{*}{$\nu = -2$} & \multirow{2}{*}{$(\bar{2}, 1)$} 
	& \multirow{2}{*}{1} & \multirow{2}{*}{0} 
	& \multirow{1}{*}{$u \equiv 1_{(4)}$} & 0 & 1
\\
 & & &  & & $u \equiv 3_{(4)}$ 
    & 1 & 0 
    \\
\hline
\multirow{4}{*}{6} & \multirow{4}{*}{$\nu = b-2 \geq -1$} & \multirow{4}{*}{$(\bar{2} :: 1)$} & \multirow{4}{*}{2} & \multirow{2}{*}{0 if $b$ is even} & $u \equiv 1_{(4)}$
    & 0 & 2 
    \\
 &  & &  &  & $u \equiv 3_{(4)}$
    & 2 & 0 
    \\
\cline{5-8}
 & &  & & \multirow{2}{*}{1 if $b$ is odd} & \multirow{2}{*}{--}
    & \multirow{2}{*}{1} & \multirow{2}{*}{1}  
    \\
 & &  & &  & 
    &   &   
    \\
\hline
\hline
7 & $\nu = 2$ & (1, 2) & 1 & 1 & -- & 1 & 1 
\\
\hline
\multirow{2}{*}{8} & \multirow{2}{*}{$\nu = 4$} & \multirow{2}{*}{(1; 2)} 
	& \multirow{2}{*}{1} & \multirow{2}{*}{0}
	&  $u \equiv 1_{(4)}$ & 1 & 2 
\\
 & & & & & $u \equiv 3_{(4)}$ & 2 & 1 
\\
\hline
\multirow{4}{*}{9} & \multirow{4}{*}{$\nu = 2b \geq 6$} & \multirow{4}{*}{(1 :: 2)} 
	& \multirow{4}{*}{2} & \multirow{2}{*}{0 if $b$ is even}
	& $u \equiv 1_{(4)}$ & 2 & 4 
\\
 & & & & & $u \equiv 3_{(4)}$ & 4 & 2
 \\
\cline{5-8}
 & &  & & \multirow{2}{*}{1 if $b$ is odd} 
 	& \multirow{2}{*}{--}  & \multirow{2}{*}{3} & \multirow{2}{*}{3} 
 \\
 & & & & & & & 
\\
\hline
\multirow{2}{*}{10} & \multirow{2}{*}{$\nu = 2$} & \multirow{2}{*}{$(1, \bar{2})$} 
	& \multirow{2}{*}{1} & \multirow{2}{*}{0} 
	& \multirow{1}{*}{$u \equiv 1_{(4)}$} & 0 & 1
\\
 & & &  & & $u \equiv 3_{(4)}$ 
    & 1 & 0 
    \\
\hline

\multirow{4}{*}{11} & \multirow{4}{*}{$\nu = 2b \geq 4$} & \multirow{4}{*}{$(1 :: \bar{2})$}  
	& \multirow{4}{*}{2} & \multirow{2}{*}{0 if $b$ is odd}
	& $u \equiv 1_{(4)}$ & 0 & 2
 \\
 & & & & & $u \equiv 3_{(4)}$ & 2 & 0 
 \\
\cline{5-8}
 & & & & \multirow{2}{*}{1 if $b$ is even}
 	& \multirow{2}{*}{--} & \multirow{2}{*}{1} & \multirow{2}{*}{1}
 \\
 & & & & & & &
\\
\hline
\hline
\end{tabular}
$$
\centerline{\bf Table 5.1: Hasse invariant Distributions when $p=2$}
\end{minipage}
\end{sideways}
}

\newcommand{\TableHasseTwo}
{
\begin{sideways}
\begin{minipage}{\textheight}
$$
\begin{tabular}{  c ||  c  |  c  ||  c   | c  || c  | c | c | }
\multirow{2}{*}{\text{\#}}   & Allowed  & Partial Local & \# of &  Trains in Valuation & \multirow{2}{*}{Cases }& \multicolumn{2}{c|}{\# of Hasse Invariants}
\\
& $\nu = \ord_2(\det_G(Q))$ & Genus Symbol & Trains & Adjustment &   & $\hspace{.13in}c_2 = 1\hspace{.13in}$ & $c_2 = -1$ 
\\
\hline 
\hline
\multirow{2}{*}{12} & \multirow{2}{*}{$\nu = 3$} & \multirow{2}{*}{(1, 1, 1)} 
	& \multirow{2}{*}{1} & \multirow{2}{*}{1}
	& \multirow{2}{*}{--} 
	& \multirow{2}{*}{1} & \multirow{2}{*}{1} 
\\
& & & & 
	&&&
\\
\hline
\multirow{2}{*}{13} & \multirow{2}{*}{$\nu = 4$} & \multirow{2}{*}{(1, 1; 1)} 
	& \multirow{2}{*}{1} & \multirow{2}{*}{1}
	& \multirow{2}{*}{--} 
	& \multirow{2}{*}{2} & \multirow{2}{*}{2} 
\\
& & & & & 
	&  &  
\\
\hline
\multirow{2}{*}{14} & \multirow{2}{*}{$\nu = 5$} & \multirow{2}{*}{(1; 1, 1)} 
	& \multirow{2}{*}{1} & \multirow{2}{*}{1}
	& \multirow{2}{*}{--}
	& \multirow{2}{*}{2} & \multirow{2}{*}{2} 
\\
& & & & & 
	&  &  
\\
\hline
\multirow{2}{*}{15} & \multirow{2}{*}{$\nu = 6$} & \multirow{2}{*}{(1; 1; 1)}  
	& \multirow{2}{*}{1} & \multirow{2}{*}{0}
	& \multirow{1}{*}{$u \equiv 1_{(4)}$}
	& 1 & 3 
\\
& & & & & $u \equiv 3_{(4)}$
	& 3 & 1 
\\
\hline
\multirow{2}{*}{16} & \multirow{2}{*}{$\nu = c + 1\geq 5$} & \multirow{2}{*}{(1, 1 :: 1)}  
	& \multirow{2}{*}{2} & \multirow{2}{*}{2} 
	& \multirow{2}{*}{--}
	& \multirow{2}{*}{4} & \multirow{2}{*}{4} 
 \\
 & & & & & 
 	& &  
\\
\hline
17 & $\nu = 2b+1 \geq 7$ & (1 :: 1, 1)  & 2 & 2 
	& -- & 4 & 4 
\\
\hline
\multirow{4}{*}{18} & \multirow{4}{*}{$\nu = 2b+2 \geq 8$} & \multirow{4}{*}{(1 :: 1; 1)}  
	& \multirow{4}{*}{2} & \multirow{2}{*}{0 if $b$ is even}
	& $u \equiv 1_{(4)}$
	& 2 & 6 
\\
& & & & & $u \equiv 3_{(4)}$
	& 6 & 2 
\\
\cline{5-8}
 & & & & \multirow{2}{*}{1 if $b$ is odd}
 	& \multirow{2}{*}{--}
	& \multirow{2}{*}{4} & \multirow{2}{*}{4} 
\\
& & & & & 
	&  &  
\\
\hline
\multirow{4}{*}{19} & \multirow{4}{*}{$\nu = c+2 \geq 7$} & \multirow{4}{*}{(1; 1:: 1)}  
	& \multirow{4}{*}{2} & \multirow{2}{*}{0 if $c$ is even}
	& $u \equiv 1_{(4)}$
	& 2 & 6 
\\
& & & & & $u \equiv 3_{(4)}$
	& 6 & 2 
\\
\cline{5-8}
 & & & & \multirow{2}{*}{1 if $c$ is odd}
 	& \multirow{2}{*}{--}
	& \multirow{2}{*}{4} & \multirow{2}{*}{4} 
\\
& & & & & 
	&  &  
\\
\hline
\multirow{4}{*}{20} & \multirow{2}{*}{$\nu = b+ c \geq 9$} & \multirow{4}{*}{(1 :: 1 :: 1)}  
	& \multirow{4}{*}{3} & \multirow{2}{*}{0 if $b$ and $c$ are even}
	& $u \equiv 1_{(4)}$
	& 4 & 12 
\\
& & & & & $u \equiv 3_{(4)}$
	& 12 & 4 
\\
\cline{5-8}
& \multirow{2}{*}{$ b \geq 3,  c \geq b+3$} &  &  & \multirow{2}{*}{2 otherwise} 
	& \multirow{2}{*}{--}
	& \multirow{2}{*}{8} & \multirow{2}{*}{8} 
\\
& & & & & & &
\\
\hline
%
%
\end{tabular}
$$
\centerline{\bf Table 5.1: Hasse invariant Distributions when $p=2$ (continued)}
\end{minipage}
\end{sideways}
}


\newcommand{\TableMassOne}
{
\begin{sideways}
\begin{minipage}{\textheight}
\vspace{-1in}
$$ \hspace{-1.5in}
\begin{tabular}{  c ||  c  |  c  ||  c  | c  | c  | c  || c || c ||  }
\multirow{2}{*}{\text{\#}}   & Allowed  
	& Partial Local
	& Species &  Diagonal 
	& Cross & Type & \multirow{1}{*}{$2$-mass}
	& Normalized Densities
\\
& $\nu = \ord_2(\det_G(Q))$ & Genus Symbol & List  & Factor 
	& Product  & Factor & $m_2(Q)$
    	& $(1- \frac{1}{2^2})\cdot\beta_{2, Q}^{-1}(Q)$ 
\\
\hline 
\hline
%
%
%
\multirow{4}{*}{1} & \multirow{4}{*}{$\nu = 0$} & \multirow{4}{*}{(3)} 
	& \multirow{2}{*}{[1, {\bf 2+}, 1] if I${}_3$-octane $\equiv \pm1_{(8)}$} 
	& \multirow{2}{*}{$\frac{1}{4}$} 
	& \multirow{4}{*}{$1$} & \multirow{4}{*}{$1$}
	& \multirow{2}{*}{$\frac{1}{4}$} 
	& \multirow{2}{*}{$\frac{3}{8}$} 
\\
&&&&&&&& \\
\cline{4-5} \cline{8-9}
&&	& \multirow{2}{*}{[1, {\bf 2-}, 1] if I${}_3$-octane $\equiv \pm3_{(8)}$} 
	& \multirow{2}{*}{$\frac{1}{12}$}
	&&
	& \multirow{2}{*}{$\frac{1}{12}$}
	& \multirow{2}{*}{$\frac{1}{8}$} \\
&&&&&&&& \\
\hline
\hline
\multirow{2}{*}{2} & \multirow{2}{*}{$\nu = 1$} & \multirow{2}{*}{(2, 1)} 
	& \multirow{2}{*}{[1, 1, 1, 1]} 
	& \multirow{2}{*}{$2^{-4}$} & \multirow{2}{*}{$2^1$} & \multirow{2}{*}{$2^1$} 
	& \multirow{2}{*}{$2^{-2}$} 
	& \multirow{2}{*}{$2^{-5} \cdot 3$}
    \\
 & &  & &&&&&
\\
\hline
\multirow{4}{*}{3} & \multirow{4}{*}{$\nu = 2$} & \multirow{4}{*}{(2; 1)} 
	& \multirow{2}{*}{[1, {\bf 1}, 1, 0, 1] if I${}_2$-oddity $\equiv 2_{(4)}$} 
	& \multirow{2}{*}{$2^{-4}$} & \multirow{4}{*}{$2^2$} & \multirow{4}{*}{$2^0$} 
	& \multirow{2}{*}{$2^{-2}$} 
	& \multirow{2}{*}{$2^{-7} \cdot 3$}
   \\
 & & & &  &&&&
    \\
\cline{4-5} \cline{8-9}
 & & &  \multirow{2}{*}{[1, {\bf 0}, 1, 0, 1] if I${}_2$-oddity $\equiv 0_{(4)}$} 
 	& \multirow{2}{*}{$2^{-3}$} &&
	& \multirow{2}{*}{$2^{-1}$}
	& \multirow{2}{*}{$2^{-6} \cdot 3$}
    \\
 & & &  & &&&&
    \\
\hline
\multirow{4}{*}{4} & \multirow{4}{*}{$\nu = b \geq 3$} & \multirow{4}{*}{(2 :: 1)} 
	& \multirow{2}{*}{[1, {\bf 1}, 1, :: 1, 0, 1] if I${}_2$-oddity $\equiv 2_{(4)}$} 
	& \multirow{2}{*}{$2^{-5}$} & \multirow{4}{*}{$2^{b}$} & \multirow{4}{*}{$2^0$} 
	& \multirow{2}{*}{$2^{b-5}$}
	& \multirow{2}{*}{$2^{-b-6} \cdot 3$}
    \\
 &  & &&&&&&
    \\
\cline{4-5} \cline{8-9}
 & &  
 	&\multirow{2}{*}{[1, {\bf 0}, 1, :: 1, 0, 1] if I${}_2$-oddity $\equiv 0_{(4)}$} 
	& \multirow{2}{*}{$2^{-4}$} & \multirow{1}{*}{}
	&  & \multirow{2}{*}{$2^{b-4}$} 
	& \multirow{2}{*}{$2^{-b-5} \cdot 3$}
    \\
 & &  &&&&&&
    \\
\hline
\multirow{2}{*}{5} & \multirow{2}{*}{$\nu = -2$} & \multirow{2}{*}{$(\bar{2}, 1)$} 
	& \multirow{2}{*}{[3, 0, 1]}  
	& \multirow{2}{*}{$\frac{1}{3}$} & \multirow{2}{*}{$2^1$} & \multirow{2}{*}{$2^{-2}$} 
	& \multirow{2}{*}{$\frac{1}{6}$} 
	& \multirow{2}{*}{$4$}
\\
&&&&&&&&\\
\hline
\multirow{4}{*}{6} & \multirow{4}{*}{$\nu = b-2 \geq -1$} & \multirow{4}{*}{$(\bar{2} :: 1)$} 
	& \multirow{2}{*}{[{\bf 2+}, :: 1, 0, 1] if II${}_2$-sign is $+$} 
	& \multirow{2}{*}{$\frac{1}{4}$} 
	& \multirow{4}{*}{$2^{b + 1}$} & \multirow{4}{*}{$2^{-2}$} 
	& \multirow{2}{*}{$2^{b-3}$} 
	& \multirow{2}{*}{$3 \cdot 2^{-b}$}
\\
&&&&&&&&\\
\cline{4-5} \cline{8-9}
&& & \multirow{2}{*}{[{\bf 2-}, :: 1, 0, 1] if II${}_2$-sign is $-$} 
	& \multirow{2}{*}{$\frac{1}{12}$} &&
	& \multirow{2}{*}{$\frac{1}{3} \cdot 2^{b-3}$} 
	& \multirow{2}{*}{$2^{-b}$}
	\\
&&&&&&&&\\
\hline
\hline
\multirow{2}{*}{7} & \multirow{2}{*}{$\nu = 2$} & \multirow{2}{*}{$(1, 2)$} 
	& \multirow{2}{*}{[1, 1, 1, 1]} 
	& \multirow{2}{*}{$2^{-4}$} & \multirow{2}{*}{$2^1$} & \multirow{2}{*}{$2^1$} 
	& \multirow{2}{*}{$2^{-2}$} 
	& \multirow{2}{*}{$2^{-7} \cdot 3$}
\\
 & & &&&&&&
\\
\hline
\multirow{4}{*}{8} & \multirow{4}{*}{$\nu = 4$} & \multirow{4}{*}{(1; 2)} 
	& \multirow{2}{*}{[1, 0, 1, {\bf 1}, 1] if I${}_2$-oddity $\equiv 2_{(4)}$}  
	& \multirow{2}{*}{$2^{-4}$} & \multirow{4}{*}{$2^{2}$} & \multirow{4}{*}{$2^0$} 
	& \multirow{2}{*}{$2^{-2}$} 
	& \multirow{2}{*}{$2^{-11} \cdot 3$}
\\
 & & &&&&&&
\\
\cline{4-5} \cline{8-9}
 & & 
 	& \multirow{2}{*}{[1, 0, 1, {\bf 0}, 1] if I${}_2$-oddity $\equiv 0_{(4)}$}  
	& \multirow{2}{*}{$2^{-3}$} & \multirow{1}{*}{} & \multirow{1}{*}{} 
	& \multirow{2}{*}{$2^{-1}$} 
	& \multirow{2}{*}{$2^{-10} \cdot 3$}
\\
 & & &&&&&&
\\
\hline
\multirow{4}{*}{9} & \multirow{4}{*}{$\nu = 2b \geq 6$} & \multirow{4}{*}{(1 :: 2)} 
	& \multirow{2}{*}{[1, 0, 1, :: 1, {\bf 1}, 1]  if I${}_2$-oddity $\equiv 2_{(4)}$} 
	& \multirow{2}{*}{$2^{-5}$} & \multirow{4}{*}{$2^{b}$} & \multirow{4}{*}{$2^0$} 
	& \multirow{2}{*}{$2^{b-5}$} 
	& \multirow{2}{*}{$2^{-3b-6} \cdot 3$}
\\
 & & & & & & & &
 \\
\cline{4-5} \cline{8-9}
 & &   
	& \multirow{2}{*}{[1, 0, 1, :: 1, {\bf 0}, 1]  if I${}_2$-oddity $\equiv 0_{(4)}$}  
	& \multirow{2}{*}{$2^{-4}$} & \multirow{1}{*}{} & 
	& \multirow{2}{*}{$2^{b-4}$} 
	& \multirow{2}{*}{$2^{-3b-5} \cdot 3$}
 \\
 & & & 
 	& \multirow{1}{*}{} & \multirow{1}{*}{} & \multirow{1}{*}{} & &
\\
\hline
\multirow{2}{*}{10} & \multirow{2}{*}{$\nu = 2$} & \multirow{2}{*}{$(1, \bar{2})$}  
	& \multirow{2}{*}{[1, 0, 3]} 
	& \multirow{2}{*}{$\frac{1}{6}$} & \multirow{2}{*}{$2^1$} & \multirow{2}{*}{$2^{-2}$} 
	& \multirow{2}{*}{$\frac{1}{6}$} 
	& \multirow{2}{*}{$2^{-6}$}
\\
&&&&&&&&\\
\hline
\multirow{4}{*}{11} & \multirow{4}{*}{$\nu = 2b \geq 4$} & \multirow{4}{*}{$(1 :: \bar{2})$}  
	& \multirow{2}{*}{[1, 0, 1, :: {\bf 2+}] if II${}_2$-sign is $+$} 
	& \multirow{2}{*}{$\frac{1}{4}$}  
	& \multirow{4}{*}{$2^b$} & \multirow{4}{*}{$2^{-2}$} 
	& \multirow{2}{*}{$2^{b-4}$} 
	& \multirow{2}{*}{$3 \cdot 2^{-3b-5}$}
\\
&&&&&&&&\\
\cline{4-5} \cline{8-9}
&&& \multirow{2}{*}{[1, 0, 1, :: {\bf 2-}] if II${}_2$-sign is $-$} 
	& \multirow{2}{*}{$\frac{1}{12}$}  &&
	& \multirow{2}{*}{$\frac{1}{3}\cdot 2^{b-4}$} 
	&  \multirow{2}{*}{$2^{-3b-5}$}
\\
&&&&&&&&\\
\hline
\hline
\end{tabular}
$$
\centerline{\bf Table 5.2:  Normalized Local Mass Factors when $p=2$}
\end{minipage}
\end{sideways}
}

\newcommand{\TableMassTwo}
{
\begin{sideways}
\begin{minipage}{\textheight}
$$\hspace{-1in}
\begin{tabular}{  c ||  c  |  c  ||  c   | c  | c  | c || c || c || }
\multirow{2}{*}{\text{\#}}   & Allowed  & Partial Local & Species &  Diagonal 
	& Cross & Type & \multirow{1}{*}{$2$-mass}
	& Normalized Densities
\\
& $\nu = \ord_2(\det_G(Q))$ & Genus Symbol & List  & Factor 
	& Product  & Factor & $m_2(Q)$
    	& $(1- \frac{1}{2^2})\cdot\beta_{2, Q}^{-1}(Q)$ 
\\
\hline 
\hline
\multirow{2}{*}{12} & \multirow{2}{*}{$\nu = 3$} & \multirow{2}{*}{(1, 1, 1)} 
	& \multirow{2}{*}{[1, 1, 1, 1, 1]} 
	& \multirow{2}{*}{$2^{-5}$} & \multirow{2}{*}{$2^2$} &  \multirow{2}{*}{$2^2$}
	& \multirow{2}{*}{$2^{-1}$}
	& \multirow{2}{*}{$2^{-8} \cdot 3$}
	\\
& & & & & & & &
\\
\hline
\multirow{2}{*}{13} & \multirow{2}{*}{$\nu = 4$} & \multirow{2}{*}{(1, 1; 1)} 
	& \multirow{2}{*}{[1, 1, 1, 1, 0, 1]} 
	& \multirow{2}{*}{$2^{-5}$} & \multirow{2}{*}{$2^3$} & \multirow{2}{*}{$2^1$} 
	& \multirow{2}{*}{$2^{-1}$}
	& \multirow{2}{*}{$2^{-10} \cdot 3$}
\\
& & & & & & & &
\\
\hline
\multirow{2}{*}{14} & \multirow{2}{*}{$\nu = 5$} & \multirow{2}{*}{(1; 1, 1)} 
	& \multirow{2}{*}{[1, 0, 1, 1, 1, 1]} 
	& \multirow{2}{*}{$2^{-5}$} & \multirow{2}{*}{$2^3$} &  \multirow{2}{*}{$2^1$} 
	& \multirow{2}{*}{$2^{-1}$}
	& \multirow{2}{*}{$2^{-12} \cdot 3$}
\\
& & & & & & & &
\\
\hline
\multirow{2}{*}{15} & \multirow{2}{*}{$\nu = 6$} & \multirow{2}{*}{(1; 1; 1)}  
	& \multirow{2}{*}{[1, 0, 1, 0, 1, 0, 1]} 
	& \multirow{2}{*}{$2^{-4}$} & \multirow{2}{*}{$2^4$} & \multirow{2}{*}{$2^0$} 
	& \multirow{2}{*}{$2^{0}$}
	& \multirow{2}{*}{$2^{-13} \cdot 3$}
\\
& & & & & & & &
\\
\hline
\multirow{2}{*}{16} & \multirow{2}{*}{$\nu = c +1 \geq 5$} & \multirow{2}{*}{(1, 1 :: 1)}  
	& \multirow{2}{*}{[1, 1, 1, 1, :: 1, 0, 1]} 
	& \multirow{2}{*}{$2^{-6}$} & \multirow{2}{*}{$2^c$} & \multirow{2}{*}{$2^1$} 
	& \multirow{2}{*}{$2^{c-5}$}
	& \multirow{2}{*}{$2^{-c-8} \cdot 3$}
 \\
 & & & & & & & &
\\ 
\hline
\multirow{2}{*}{17} & \multirow{2}{*}{$\nu = 2b+1 \geq 7$} & \multirow{2}{*}{(1 :: 1, 1)} 
	& \multirow{2}{*}{[1, 0, 1, :: 1, 1, 1, 1]} 
	& \multirow{2}{*}{$2^{-6}$} & \multirow{2}{*}{$2^{b+1}$} & \multirow{2}{*}{$2^1$} 
	& \multirow{2}{*}{$2^{b-4}$}
	& \multirow{2}{*}{$2^{-3b-7} \cdot 3$}
\\
& & & & & & & &
\\
\hline
\multirow{2}{*}{18} & \multirow{2}{*}{$\nu = 2b+2 \geq 8$} & \multirow{2}{*}{(1 :: 1; 1)}  
	& \multirow{2}{*}{[1, 0, 1, :: 1, 0, 1, 0, 1]} 
	& \multirow{2}{*}{$2^{-5}$} & \multirow{2}{*}{$2^{b+2}$} & \multirow{2}{*}{$2^0$} 
	& \multirow{2}{*}{$2^{b-3}$}
	& \multirow{2}{*}{$2^{-3b-8} \cdot 3$}
	\\
& & & & & & & &
\\
\hline
\multirow{2}{*}{19} & \multirow{2}{*}{$\nu = c+2 \geq 7$} & \multirow{2}{*}{(1; 1 :: 1)}  
	& \multirow{2}{*}{[1, 0, 1, 0, 1, :: 1, 0, 1]} 
	& \multirow{2}{*}{$2^{-5}$} & \multirow{2}{*}{$2^c$} & \multirow{2}{*}{$2^0$} 
	& \multirow{2}{*}{$2^{c-5}$}
	& \multirow{2}{*}{$2^{-c-10} \cdot 3$}
\\
& & & & & & & &
\\
\hline
\multirow{2}{*}{20} & \multirow{2}{*}{$\nu = b+ c \geq 9$} & \multirow{2}{*}{(1 :: 1 :: 1)}  
	& \multirow{2}{*}{[1, 0, 1, :: 1, 0, 1, :: 1, 0, 1]} 
	& \multirow{2}{*}{$2^{-6}$} & \multirow{2}{*}{$2^c$}  & \multirow{2}{*}{$2^0$} 
	& \multirow{2}{*}{$2^{c-6}$}
	& \multirow{2}{*}{$2^{-2b-c-7} \cdot 3$}
\\
& & & & & & & &
\\
\hline
%
%
\end{tabular}
$$
\centerline{\bf Table 5.2: Normalized Local Mass Factors when $p=2$ (continued)}
\end{minipage}
\end{sideways}
}


\newcommand{\TableEulerOne}
{
\begin{sideways}
\begin{minipage}{\textheight}
$$
\begin{tabular}{  c ||  c  |  c  ||  c   | c  || c  | c | c || c || }
\multirow{2}{*}{\text{\#}}   & Allowed  & \multirow{2}{*}{PLGS} & \multirow{2}{*}{$4 \cdot A^*_{2^\nu}$} & \multirow{2}{*}{$4 \cdot B^*_{2^\nu}$} &  \multirow{1}{*}{Cases determining} 
	& \multicolumn{2}{c||}{$c_2 =$} 
	& Normalized Densities
\\
& $\nu = \ord_2(\det_G(Q))$ &  & & 
	& the $c_2$ distribution
	&  $+1$ &  $-1$ 
	& $(1- \frac{1}{2^2})\cdot\beta_{2, Q}^{-1}(Q)$
\\
\hline 
\hline
%
%
%
\multirow{2}{*}{1} & \multirow{2}{*}{$\nu = 0$} & \multirow{2}{*}{(3)} 
	& \multirow{2}{*}{$2^{-1}$}
	& $-2^{-2}$ 
	& $u \equiv 1_{(4)}$ 
	& \multirow{2}{*}{1} & \multirow{2}{*}{1}
	&  \multirow{2}{*}{$\frac{3}{8}$ or $\frac{1}{8}$}
\\
\cline{5-6}
&&&
	& $2^{-2}$
	& $u \equiv 3_{(4)}$ 
	&
	&
	&
\\
\hline
\hline
\multirow{1}{*}{2} & \multirow{1}{*}{$\nu = 1$} & \multirow{1}{*}{(2, 1)}  
	& \multirow{1}{*}{$2^{-4} \cdot 3$} & \multirow{1}{*}{0}
	& \multirow{1}{*}{--} & 1 & 1 
	& \multirow{1}{*}{$2^{-5} \cdot 3$}
 \\
\hline
\multirow{2}{*}{3} & \multirow{2}{*}{$\nu = 2$} & \multirow{2}{*}{(2; 1)} 
	& \multirow{2}{*}{$2^{-5} \cdot 3$} 
	&  \multirow{1}{*}{$-2^{-6}\cdot 3$} 
	& $u \equiv 1_{(4)}$ 
    & 1 & 2 
    & \multirow{2}{1.5in}{$2^{-7} \cdot 3$ (twice) or $2^{-6} \cdot 3$ (once)}
    \\
\cline{5-8}
 & & &  
 	& \multirow{1}{*}{$2^{-6}\cdot 3$} 
	& $u \equiv 3_{(4)}$ 
    & 2 & 1 
    &
    \\
\hline
\multirow{4}{*}{4} & \multirow{4}{*}{$\nu = b \geq 3$} & \multirow{4}{*}{(2 :: 1)} 
	& \multirow{4}{*}{$2^{-b-3} \cdot 3$} 
	&  \multirow{1}{*}{$-2^{-b-4} \cdot 3$} & \multirow{1}{*}{if $b$ is even and $u \equiv 1_{(4)}$} 
    & \multirow{1}{*}{2} & \multirow{1}{*}{4} 
    & \multirow{4}{1.8in}{$2^{-b-6} \cdot 3$ (II-oddity $\equiv 2_{(4)}$) or $2^{-b-5} \cdot 3$ (II-odd. $\equiv 0_{(4)}$)}
    \\
\cline{5-8}
 &  &  &  & $2^{-b-4} \cdot 3$ & \multirow{1}{*}{if $b$ is even and $u \equiv 3_{(4)}$}
    & 4 & 2
    &
    \\
\cline{5-8}
 & &  & & \multirow{2}{*}{$0$} & \multirow{2}{*}{if $b$ is odd}
    &  \multirow{2}{*}{3} & \multirow{2}{*}{3} 
    &
    \\
 & &  & &  & 
    & &
    &
        \\
\hline
\multirow{2}{*}{5} & \multirow{2}{*}{$\nu = -2$} & \multirow{2}{*}{$(\bar{2}, 1)$} 
	& \multirow{2}{*}{$4$} 
	&  \multirow{1}{*}{$-4$} 
	& $u \equiv 1_{(4)}$ 
    & 0 & 1 
    & \multirow{2}{*}{$4$}
    \\
\cline{5-8}
 & & &  
 	& \multirow{1}{*}{$4$} 
	& $u \equiv 3_{(4)}$ 
    & 1 & 0 
    &
    \\
\hline

\multirow{4}{*}{6} & \multirow{4}{*}{$\nu = b-2 \geq -1$} & \multirow{4}{*}{$(\bar{2} :: 1)$} 
	& \multirow{4}{*}{$2^{2-b}$} 
	&  \multirow{1}{*}{$-2^{2-b}$} & \multirow{1}{*}{if $b$ is even and $u \equiv 1_{(4)}$} 
    & \multirow{1}{*}{0} & \multirow{1}{*}{2} 
    & \multirow{4}{1.5in}{$3\cdot 2^{-b}$ (II-sign = $+$)  or $2^{-b}$ (II-sign = $-$)}
    \\
 \cline{5-8}
 &  & &  &   $2^{2-b}$ & \multirow{1}{*}{if $b$ is even and $u \equiv 3_{(4)}$}
    &  2 & 0
    &
    \\
\cline{5-8}
 & &  & & \multirow{1}{*}{$-2^{1-b}$} & \multirow{1}{*}{if $b$ is odd and $u \equiv 1_{(4)}$}
    &  \multirow{2}{*}{1} & \multirow{2}{*}{1} 
    &
    \\
 \cline{5-6}
 & &  & & $2^{1-b}$ & \multirow{1}{*}{if $b$ is odd and $u \equiv 3_{(4)}$} 
    & &
    &
        \\
\hline
\hline
\multirow{1}{*}{7} & \multirow{1}{*}{$\nu = 2$} & \multirow{1}{*}{(1, 2)}  
	& \multirow{1}{*}{$2^{-6} \cdot 3$} & \multirow{1}{*}{0}
	& \multirow{1}{*}{--} & 1 & 1 
	& \multirow{1}{*}{$2^{-7} \cdot 3$}
 \\
\hline
\multirow{2}{*}{8} & \multirow{2}{*}{$\nu = 4$} & \multirow{2}{*}{(1; 2)} 
	& \multirow{2}{*}{$2^{-9} \cdot 3$} 
	& \multirow{1}{*}{$-2^{-10} \cdot 3$}
	&  $u \equiv 1_{(4)}$ & 1 & 2 
	& \multirow{2}{1.8in}{$2^{-11} \cdot 3$ ($I_2$-oddity $\equiv 2_{(4)}$) or $2^{-10} \cdot 3$ ($I_2$-odd. $\equiv 0_{(4)}$)}
\\
\cline{5-8}
 & & & 
 	& \multirow{1}{*}{$2^{-10} \cdot 3$} 
	& $u \equiv 3_{(4)}$ & 2 & 1 
	&
\\
\hline
\multirow{4}{*}{9} & \multirow{4}{*}{$\nu = 2b \geq 6$} & \multirow{4}{*}{(1 :: 2)} 
	& \multirow{4}{*}{$2^{-3b-3} \cdot 3$} 
	& \multirow{2}{*}{$0$}
	&  \multirow{2}{*}{if $b$ is odd} 
	 & \multirow{2}{*}{3} & \multirow{2}{*}{3} 
	 &  \multirow{4}{1.9in}{$2^{-3b-6} \cdot 3$  
	 ($I_2$-oddity $\equiv 2_{(4)}$) or $2^{-3b-5} \cdot 3$  ($I_2$-odd. $\equiv 0_{(4)}$) }
\\
 & & & & & 
 	&  & 
	&
 \\
\cline{5-8}
 & &  & & \multirow{1}{*}{$-2^{-3b-4} \cdot 3$} 
 	& \multirow{1}{*}{if $b$ is even and $u \equiv 1_{(4)}$} 
	& \multirow{1}{*}{2} & \multirow{1}{*}{4} 
	&
 \\
\cline{5-8}
 & & & &  \multirow{1}{*}{$2^{-3b-4} \cdot 3$}
 	& \multirow{1}{*}{if $b$ is even and $u \equiv 3_{(4)}$}
 	&  4 & 2 
	&
\\
\hline
\multirow{2}{*}{10} & \multirow{2}{*}{$\nu = 2$} & \multirow{2}{*}{$(1, \bar{2})$}  
	& \multirow{2}{*}{$2^{-6}$} & \multirow{1}{*}{$-2^{-6}$}
	& \multirow{1}{*}{$u \equiv 1_{(4)}$} & \multirow{1}{*}{0} & \multirow{1}{*}{1} 
	& \multirow{2}{*}{$2^{-6}$}
 \\
 \cline{5-8}
 &&&& $2^{-6}$ & $u \equiv 3_{(4)}$
 & 1 & 0 & \\
\hline

\multirow{4}{*}{11} & \multirow{4}{*}{$\nu = 2b \geq 4$} & \multirow{4}{*}{$(1 :: \bar{2})$}  
	& \multirow{4}{*}{$2^{-3b-3}$} & $-2^{-3b-3}$
	& if $b$ is odd and $u \equiv 1_{(4)}$ & 0 & 2 
	& \multirow{4}{1.8in}{$3 \cdot 2^{-3b-5}$ (II-sign = $+$) or $2^{-3b-5}$ (II-sign = $-$)}
 \\
\cline{5-8}
 & & & & $2^{-3b-3}$ & if $b$ is odd and $u \equiv 3_{(4)}$ & 2 & 0 &
 \\
\cline{5-8}
 & & & & \multirow{1}{*}{$-2^{-3b-4}$}
 	& if $b$ is even and $u \equiv 1_{(4)}$ 
	& \multirow{2}{*}{1} & \multirow{2}{*}{1} &
 \\
 \cline{5-6}
 & & & & $2^{-3b-4}$ & if $b$ is even and $u \equiv 3_{(4)}$ & & &
\\
\hline
\hline
\end{tabular}
$$
\centerline{\bf Table 5.3: Partial Euler Factors when $p=2$}
\end{minipage}
\end{sideways}
}

\newcommand{\TableEulerTwo}
{
\begin{sideways}
\begin{minipage}{\textheight}
$$\hspace{-1in}
\begin{tabular}{  c ||  c  |  c  ||  c   | c  || c  | c | c || c || }
\multirow{2}{*}{\text{\#}}   & Allowed  & \multirow{2}{*}{PLGS} 
	& \multirow{2}{*}{$4 \cdot A^*_{2^\nu}$} &  \multirow{2}{*}{$4 \cdot B^*_{2^\nu}$} 
	& \multirow{1}{*}{Cases determining} &  \multicolumn{2}{c||}{$c_2 =$} 
	& Normalized Densities
\\
& $\nu = \ord_2(\det_G(Q))$ &  &   & 
	& the $c_2$ distribution
	& $+1$ & $-1$ 
    	& $(1- \frac{1}{2^2})\cdot\beta_{2, Q}^{-1}(Q)$ 
\\
\hline 
\hline
\multirow{2}{*}{12} & \multirow{2}{*}{$\nu = 3$} & \multirow{2}{*}{(1, 1, 1)} 
	& \multirow{2}{*}{$2^{-7} \cdot 3$} 
	& \multirow{2}{*}{$0$} & \multirow{2}{*}{--}
	&  \multirow{2}{*}{1} & \multirow{2}{*}{1}
	& \multirow{2}{*}{$2^{-8} \cdot 3$}
	\\
& & & 
	&&&&&
\\
\hline

\multirow{2}{*}{13} & \multirow{2}{*}{$\nu = 4$} & \multirow{2}{*}{(1, 1; 1)} 
	& \multirow{2}{*}{$2^{-8} \cdot 3$} 
	& \multirow{2}{*}{$0$} & \multirow{2}{*}{--}
	& \multirow{2}{*}{2} 
	& \multirow{2}{*}{2}
	& \multirow{2}{*}{$2^{-10} \cdot 3$}
\\
& & & 
	& & 
	&  &  & 
\\
\hline
\multirow{2}{*}{14} & \multirow{2}{*}{$\nu = 5$} & \multirow{2}{*}{(1; 1, 1)} 
	& \multirow{2}{*}{$2^{-10} \cdot 3$} 
	& \multirow{2}{*}{$0$} & \multirow{2}{*}{--}
	&  \multirow{2}{*}{2} & \multirow{2}{*}{2}
	& \multirow{2}{*}{$2^{-12} \cdot 3$}
\\
& & & 
	& & & & &
\\
\hline
\multirow{2}{*}{15} & \multirow{2}{*}{$\nu = 6$} & \multirow{2}{*}{(1; 1; 1)}  
	& \multirow{2}{*}{$2^{-11} \cdot 3$} 
	& \multirow{1}{*}{$-2^{-12} \cdot 3$} & \multirow{1}{*}{if $u\equiv 1_{(4)}$} 
	& \multirow{1}{*}{1} & \multirow{1}{*}{3}
	& \multirow{2}{*}{$2^{-13} \cdot 3$}
\\
\cline{5-8}
 & & & 
	& \multirow{1}{*}{$2^{-12} \cdot 3$} & \multirow{1}{*}{if $u\equiv 3_{(4)}$} 
	& 3 & 1 & 
\\
\hline
\multirow{2}{*}{16} & \multirow{2}{*}{$\nu = c +1 \geq 5$} & \multirow{2}{*}{(1, 1 :: 1)}  
	& \multirow{2}{*}{$2^{-c-5} \cdot 3$} 
	& \multirow{2}{*}{$0$} & \multirow{2}{*}{--}
	& \multirow{2}{*}{4} & \multirow{2}{*}{4}
	& \multirow{2}{*}{$2^{-c-8} \cdot 3$}
 \\
 & & & 
 	& & & & &
\\ 
\hline
\multirow{2}{*}{17} & \multirow{2}{*}{$\nu = 2b+1 \geq 7$} & \multirow{2}{*}{(1 :: 1, 1)} 
	& \multirow{2}{*}{$2^{-3b-4} \cdot 3$} 
	& \multirow{2}{*}{$0$} 
	& \multirow{2}{*}{--} 
	& \multirow{2}{*}{4} & \multirow{2}{*}{4}
	& \multirow{2}{*}{$2^{-3b-7} \cdot 3$}
\\
& & & & & & & &
\\
\hline
\multirow{4}{*}{18} & \multirow{4}{*}{$\nu = 2b+2 \geq 8$} & \multirow{4}{*}{(1 :: 1; 1)}  
	& \multirow{4}{*}{$2^{-3b-5} \cdot 3$} 
	& \multirow{1}{*}{$-2^{-3b-6} \cdot 3$} 
	& \multirow{1}{*}{if $b$ is even and $u \equiv 1_{(4)}$}  
	& \multirow{1}{*}{2} & \multirow{1}{*}{6}
	& \multirow{4}{*}{$2^{-3b-8} \cdot 3$}
	\\
\cline{5-8}
& & & & \multirow{1}{*}{$2^{-3b-6} \cdot 3$} 
	& \multirow{1}{*}{if $b$ is even and $u \equiv 3_{(4)}$}
	& 6 & 2 &
\\
\cline{5-8}
& & & & \multirow{2}{*}{$0$} 
	& \multirow{2}{*}{if $b$ is odd}
	& \multirow{2}{*}{4} & \multirow{2}{*}{4} &
\\
& & & & & 
	& & &
\\
\hline
\multirow{4}{*}{19} & \multirow{4}{*}{$\nu = c+2 \geq 7$} & \multirow{4}{*}{(1; 1 :: 1)}  
	& \multirow{4}{*}{$2^{-c-7} \cdot 3$} 
	& \multirow{1}{*}{$-2^{-c-8} \cdot 3$} 
	& \multirow{1}{*}{if $c$ is even and $u \equiv 1_{(4)}$} 
	& \multirow{1}{*}{2} & \multirow{1}{*}{6}
	& \multirow{4}{*}{$2^{-c-10} \cdot 3$}
\\
\cline{5-8}
& & & & \multirow{1}{*}{$2^{-c-8} \cdot 3$} 
	& \multirow{1}{*}{if $c$ is even and $u \equiv 3_{(4)}$}
	& 6 & 2 &
\\
\cline{5-8}
& & & & \multirow{2}{*}{$0$} 
	& \multirow{2}{*}{if $c$ is odd}
	& \multirow{2}{*}{4} & \multirow{2}{*}{4}
	&
\\
& & & & & 
	& & &
\\
\hline
\multirow{4}{*}{20} & \multirow{4}{1in}{$\nu = b+ c \geq 9$, $c\geq b+3$, $b\geq 3$} & \multirow{4}{*}{(1 :: 1 :: 1)}  
	& \multirow{4}{*}{$2^{-2b-c-3} \cdot 3$} 
	& \multirow{1}{*}{$-2^{-2b-c-4} \cdot 3$} 
	& \multirow{1}{*}{if $b$ and $c$ even and $u \equiv 1_{(4)}$}  
	& \multirow{1}{*}{4} & \multirow{1}{*}{12}
	& \multirow{4}{*}{$2^{-2b-c-7} \cdot 3$}
\\
\cline{5-8}
& & & & \multirow{1}{*}{$2^{-2b-c-4} \cdot 3$}
	& \multirow{1}{*}{if $b$ and $c$ even and $u \equiv 3_{(4)}$} 
	& 12 & 4 &
\\
\cline{5-8}
& & & & \multirow{2}{*}{$0$}
	& \multirow{2}{*}{if $b$ and $c$ not both even}
	& \multirow{2}{*}{8} & \multirow{2}{*}{8}
	&
\\
& & & & & & & &
\\
\hline
%
%
\end{tabular}
$$
\centerline{\bf Table 5.3: Partial Euler Factors when $p=2$ (continued)}
\end{minipage}
\end{sideways}
}

\section{Introduction and Notation}

\subsection{Introduction}

The {\bf mass} of a positive definite integer-valued quadratic form $Q$ is an important local invariant of a genus of quadratic forms that is closely related to the class number of $Q$, defined as the positive rational number 
$$
\Mass(Q) := \sum_{[Q'] \in \Gen(Q)} \frac{1}{|\Aut(Q')|}  
$$
obtained by summing the reciprocals of the sizes of the automorphism groups $\Aut(Q')$ of the ($\Z$-equivalence) classes $[Q']$ in the genus $\Gen(Q)$.  It is known that that the mass of a given quadratic form $\Q$ can be computed as an infinite product of factors called ``local densities'', though the formulas for these local densities are rather involved to evaluate and very prone to errors.  
With an eye towards arithmetic applications to recent ``discriminant-preserving'' correspondences of Bhargava, in \cite{Ha-masses_of_varying_det} we studied the seemingly more tractable problem of how to compute the sum of all masses of quadratic forms of a given (Hessian) determinant squareclass $\det_H(Q) = S$, called the primitive total mass, and to understand how this total mass grows as $S \rightarrow \infty$.  More precisely, if we let $\mathbf{Cls}(S;n)$ 
(resp. $\mathbf{Cls}^*(S;n)$) 
denote the classes of 
(resp. primitive) 
positive definite integer-valued quadratic forms over $\Z$ in $n$ variables, then the {\bf total mass} $\mathrm{TMass}_n(S)$
 and {\bf primitive total mass} $\mathrm{TMass}_n^*(S)$ 
 are 
respectively defined as 
$$
\mathrm{TMass}_n(S)
:= 
\sum_{
	[Q] \in \mathbf{Cls}(S;n)
}
\frac{1}{|\Aut(Q)|}
\qquad \text{ and }\qquad
\mathrm{TMass}^*_n(S) 
:= 
\sum_{
    [Q] \in \mathbf{Cls}^*(S;n)
}
\frac{1}{|\Aut(Q)|}.
$$

In this paper we carry out the local computations described in \cite{Ha-masses_of_varying_det} for ternary quadratic lattices and our main result in the case where $F = \Q$ gives an explicit formula for the 
 associated formal {\bf total mass Dirichlet series}
$$
D_{\mathrm{Mass}; n=3}(s)
 := 
\sum_{N \in \N} 
\frac{\mathrm{TMass}_{n=3}(N)}{N^s}.
$$
%
%
Our main result, stated when $F = \Q$ for simplicity, gives a surprisingly simple  formula for this Dirichlet series.

\begin{thm}[Theorem \ref{Thm:Formal_mass_series_over_F}, 
	Cor \ref{Cor:Formal_mass_series_over_QQ}] \label{Thrm:Main_mass_Theorem}
The total mass Dirichlet series for positive definite integer-valued ternary quadratic forms is given by 
$$
D_{\mathrm{Mass}; n=3}(s)
= 
\frac{1}{48 \cdot 2^s} 
\cdot
\[\zeta(s-1) \zeta(2s-1) -  \zeta(2s-2) \zeta(s)\].
$$
This gives the simple divisor sum formula 
$$
\mathrm{TMass}_{n=3}(S) 
= 
\tfrac{1}{48} 
\sum_{\substack{S/2 = a\cdot b^2 \\ a,b>0}} \(ab - b^2\).
$$
\end{thm}
\noindent
We actually show that similar explicit formulas holds if we consider rank 3 quadratic lattices of any fixed totally definite signature over a totally real number field $F$ where $p=2$ splits completely, and derive this from a similar result for the primitive total mass Dirichlet series whose coefficients are given by $\mathrm{TMass}^*_{n=3}(N)$.

\smallskip
To prove these results, we continue the work of \cite{Ha-masses_of_varying_det} and 
compute the formal Dirichlet series of the 
closely related quantity called the {\bf primitive total non-archimedean mass} $M^*_{n}(S)$ for any (non-archimedean Hessian) determinant squareclass $S$ when $n=3$.
When $n$ is odd that paper shows the following two structural theorems (stated here for positive definite forms where $F = \Q$):
\begin{thm}[{\cite[Cor 4.15, p20]{Ha-masses_of_varying_det}}]   \label{Thm:Eulerian_decomposition}
When $n$ is odd and we consider positive definite forms, the formal Dirichlet series 
$$
D_{M^*; n}(s) := \sum_{S\in\N} \frac{M^*_n(S)}{S^{s}}
$$
can be written as a sum
$$
D_{M^*; n}(s) = \kappa_n \cdot \[ D_{A^*; n}(s) + D_{B^*; n}(s) \] 
$$
of two Eulerian Dirichlet series $D_{A^*; n}(s)$ and $D_{B^*; n}(s)$, with some explicit constant $\kappa_n$.  
\end{thm}
\begin{thm}[{\cite[Cor 5.5, p22]{Ha-masses_of_varying_det}}]
When $n$ is odd, the Euler factors at $p$ in the Dirichlet series $D_{A^*; n}(s)$ and $D_{B^*; n}(s)$ above are each 
rational functions in $p^{-s}$.
\end{thm}
\noindent 
That paper also explicitly computes these Euler factors when $n=2$ under the assumption that $p=2$ splits completely in $F$. 
%
%
%
In this paper we continue this work to cover the case when $n=3$, by computing the Euler factors of $D_{A^*; n=3}(s)$ and $D_{B^*; n=3}(s)$ at all primes $\p$ of any number field $F$ where $p=2$ splits completely.  
The $n=3$ case is necessary for understanding the 2-parts of class groups of $n$-monogenic cubic fields
which will be described in \cite{Bh-Ha-Sh}.
We prove these results by using the theory of local genus invariants to enumerate the contributions of each local genus to these Euler factors by grouping the local genera according to their Jordan block structure (when $\p\nmid 2$) or the partial local genus symbol (when $\p\mid2$ and $F_\p = \Q_2$) as described in \cite[Defn 5.1, p22]{Ha-masses_of_varying_det}.  This enumeration is particularly difficult when $\p\mid 2$, as here there are 20 different cases to consider and also the theory of local genera for these primes is more involved (e.g. compare Tables \ref{Tab:odd_p_invariants} and \ref{Tab:odd_p_Euler_info} with Tables 5.1-5.3 in Section \ref{Sec:p=2_tables}).  
Our second main theorem, stated for simplicity when $F=\Q$, states that 
\begin{thm}[Theorem \ref{Thm:Explicit_Dirichlet_Series_Formula}, 
	Cor \ref{Cor:Explicit_Dirichlet_Series_Formula_over_QQ}] 
	\label{Thrm:Main_AB_Theorem}
Suppose that $F=\Q$, $n=3$ and we consider positive definite quadratic forms, then we have $\kappa_3 = 2\zeta(2) = \frac{\pi^2}{3}$,
$$
D_{A^*; n=3}(s) = \frac{ \zeta(s+1) \zeta(2s+3)}{2^{s} \cdot \zeta(3s+6)},
\qquad\text{ and }\qquad
D_{B^*; n=3}(s) 
= \frac{(-1) \cdot  \zeta(2s+2) \zeta(s+2)}{2^{s} \cdot \zeta(3s+6)}.
$$
\end{thm}

\noindent
This gives an explicit formula for $D_{M^*; n=3}(s)$ by Theorem \ref{Thm:Eulerian_decomposition}.
To describe the Euler factors  for $D_{A^*; n=3}(s)$ and $D_{B^*; n=3}(s)$ we first compute the simpler local Dirichlet series  described in \cite[Thrm 5.4, p22]{Ha-masses_of_varying_det} where we vary the Hessian determinant squareclass $S_\p$ but fix its normalized squareclass $\widetilde{S_\p}$.  From these results we then derive the desired Euler factors by analyzing how this normalized squareclass $\widetilde{S_\p}$ varies within our chosen family of distinguished squareclasses.

\smallskip
Finally, we remark on how the formulas in our main theorems can be numerically verified in the case of positive definite forms when $F=\Q$, and the specialized software \cite{Ha-Sage-Dirichlet, Ha-Sage-Tables, Ha-Sage-Mass} we  developed to perform this verification for $\Z$-valued positive definite quadratic forms of Hessian determinant $\leq 20,000$.

\medskip
{\bf Acknowledgements: }
The author would like to thank Manjul Bhargava for posing a question that led to this work, and MSRI for their hospitality during their Spring 2011 semester in Arithmetic Statistics.  The author would also like to warmly thank Robert Varley for his continuing interest in this work, and for several helpful conversations.
This work was completed at the University of Georgia between December 2009 and Summer 2011, and was partially supported by the NSF Grant DMS-0603976.

\subsection{Notation}
Our notation is consistent with that of \cite{Ha-masses_of_varying_det}, and any definitions not explicitly given here can be found there.  For convenience, we now recall some of the relevant definitions and notation there.
Throughout this paper we let $\Z := \{\cdots, -2, -1, 0, 1, 2, \cdots\}$ denote the integers,  $\Q$ the rational numbers, $\R$ the real numbers, $\C$ the complex numbers, and $\N$ the natural numbers (i.e. positive integers).   We also denote the units (i.e. invertible elements) of a ring $R$ by $R^\times$.

\smallskip
{\bf Number Fields:}
We let $F$ denote a number field with ring of integers $\OF$, $v$ is a place of $F$, $\p$ is a prime ideal (or more simply, a {\bf prime}) of $F$, $F_v$ is the completion of $F$ at $v$, $\Ov$ is the ring of integers of $F_v$ (which is $F_v$ itself when $v$ is archimedean).  
We denote by $\infty$ the archimedean place of the rational numbers $\Q$, and write $v\mid\infty$ to denote that $v$ in archimedean.  We also identify the conjugate embeddings $v$ where $F_v = \C$. We denote the set of non-archimedean (finite) places (i.e. primes $\p$) of $F$ by $\mathbf{f}$.  
For any finite set $\T \subset \mathbf{f}$ we let $I^{\T}(\OF)$ denote the set of invertible (integral) ideals of $\OF$, relatively prime to all $\p\in\T$.  We also adopt the general convention that quantities denoted by Fraktur letters (e.g. $\p, \mathfrak{a}, \n$, etc.) will be (possibly fractional) ideals of $F$.

We denote by $\A_{F, \mathbf{f}}^\times := \prod'_{\p \in \mathbf{f}} F_\p^\times$ the non-archimedean ideles of $F$, where the restricted direct product $\prod'$ requires that all but finitely many components lie in $\Op^\times$.  For convenience we will often write products $\prod_{\p\in\mathbf{f}}$ more simply as unquantified products $\prod_\p$, and similarly write $\prod_v$ for a product over all places $v$ of $F$.

\smallskip
{\bf Squareclasses:}
We let $\SqCl(\Op^\times):= (\Op^\times)/(\Op^\times)^2$ and $\SqFpInt := (F_\p^\times)/(\Op^\times)^2$ denote the local integral and local unit squareclasses at the prime $\p$ of $F$. 
We define the {\bf valuation at $\p$} of a local integral squareclass $S$ by the expression $\ord_\p(S) =  \ord_\p(S) := \ord_\p(\al)$ when  $S = \al(\Op^\times)^2$ for some $\al \in F_\p^\times$, and define the {\bf (local) ideal} $\I(S) := \p^{\ord_\p(S)}$ associated to $S$.

\smallskip
{\bf Quadratic Forms:}
When $\p\mid 2$ and $F_\p = \Q_2$ we will use the terminology {\bf train}, {\bf compartments}, {\bf oddity}, and {\bf sign} describing the local genus symbol as defined in \cite[Ch 15, \textsection7.3-7.5, pp380-382]{CS-book}.
We define the {\bf overall sign} of a local genus symbol at $\p \mid 2$ where $F_\p = \Q_2$ as the product of the signs of every train.
We denote by $\mathbf{Cls}^*(S, \vec{\sigma}_\infty, \c_{\S}; n)$ the set of $\OF$-equivalence classes $[Q]$ of primitive $\OF$-valued rank $n$ quadratic $\OF$-lattices  with fixed signature $\sigma_v(Q) = (\vec{\sigma}_\infty)_v$ at all places $v \mid \infty$, fixed Hasse invariants $c_\p(Q) = (\c_\S)_\p$ at the finitely many primes $\p\in\S$, and Hessian determinant $\det_H(G) = S$.  We also denote the {\bf genus} of $Q$ by $\Gen(Q)$.

We often refer to the $n=3$ case of various quantities defined in \cite{Ha-masses_of_varying_det}, and we frequently omit the $n=3$ subscript to simplify our notation (e.g. $A_\p^*:= A^*_{\p;n=3}$, etc.).


\bigskip



\section{Local ternary computations at primes $\p\nmid 2$}\label{Section:p_odd}

In this section we perform the local computations needed to compute the local Euler factors $A^*_\p(S)$ and $B^*_\p(S)$ of Theorem \ref{Thm:Eulerian_decomposition}, computing the relevant Hasse invariants $c_\p$ and associated normalized (reciprocal) local densities $(1- \frac{1}{q^2})\cdot\beta_{Q, \p}(Q)^{-1}$.
To do this, we classify local genera first in terms of their Jordan block structures, and then by the allowed choices of units giving rise to canonical local genus representatives.  

For convenience, when $\p\nmid 2$  we define $\ve_i := \leg{u_i}{\p}$ for $1\leq i \leq 3$ and let $\ve' := \leg{-1}{\p} = (-1)^\frac{q-1}{2}$.  We also take $\pi_\p \in \SqFpInt$ to be a fixed uniformizing squareclass (i.e. where $\ord_\p(\pip) = 1$).


\begin{lem}  \label{lemma:tuple}
For primes $\p\nmid2$, the $\Op$-inequivalent integer-valued (non-degenerate) primitive ternary quadratic forms $Q_\p$ are in bijection with the tuples $T_\p := (\p; a, b; u_1, u_2, u_3)$ where $a,b\in \Z$ with $0\leq a\leq b$,  and the $u_i \in \Op^\times/(\Op^\times)^2$ are freely chosen under the constraints:
\begin{align}
&0= a = b \implies u_1 = u_2 = 1, \\
&0 = a \neq b  \implies u_1 = 1, \\
&0 \neq a = b \implies u_2 = 1.
\end{align}
Those $Q_\p$ with Gram determinant squareclass $S \in \SqFpInt$ where $S = \pip^\nu u$ with $\nu := \ord_\p(S)$ and $u\in \Op^\times$ correspond exactly to those tuples $T_\p$ as above for which $a+b=\nu$ and $u_1 u_2 u_3 = u$.
\end{lem}

\begin{proof}
This is the Jordan decomposition for quadratic forms over $\Op$ \cite[\textsection92:1-2, pp246-7]{OM}, 
which says that for $\p\nmid 2$ any quadratic form $Q_\p$ can be diagonalized over $\Op$, and that two such diagonal forms are isomorphic if their corresponding $\p$-power scale Jordan components have the same determinant squareclass and dimension for all $\p$-powers.  Here we arrange the variables in increasing $\p$-power order, and note that our primitivity assumption means that the first variable coefficient is not divisible by $\p$.  Our bijection is given by the explicit normalization $Q_\p \sim_{\Op} u_1x^2 + \pip^a u_2 y^2 + \pip^b u_3 z^2$.
\end{proof}

\begin{lem} \label{lemma:Hasse}
The Hasse invariant $c_\p$ (defined by $c_\p := \prod_{i<j} (a_i, a_j)_\p$ for $\sum_i a_i x_i^2$) of a genus of ternary quadratic forms associated to the tuple $T_\p$ at a prime $\p\nmid 2$ is given by
$$
c_\p = \leg{u_1}{\p}^{a+b} \leg{u_2}{\p}^{b} \leg{u_3}{\p}^{a} \leg{-1}{\p}^{ab}
= \ve_1^{a+b} \cdot \ve_2^{b} \cdot \ve_3^{a} \cdot (\ve')^{ab}.
$$
In particular, we have
\begin{align}
0 = a = b  &\quad\implies\quad  c_\p = 1, \\
0 = a < b  &\quad\implies\quad  c_\p = \ve_2^b, \\
0 < a = b  &\quad\implies\quad  c_\p = \leg{-1}{\p}^b \ve_3^b.
\end{align}
\end{lem}

\begin{proof}
From the definition of $c_\p$ as a product of Hilbert symbols, we have 
\begin{align*}
c_\p 
&=  \prod_{i<j} (a_i, a_j)_\p \\
&= (u_1, \pip^a u_2)_\p \, (u_1, \pip^b u_3)_\p \, (\pip^a u_2, \pip^b u_3)_\p \\
&= \cancel{(u_1, u_2)_\p} \, (u_1, \pip^a)_\p \, \cancel{(u_1, u_3)_\p} \, (u_1, \pip^b)_\p 
\, (\pip^a u_2, \pip^b u_3)_\p \\
&=  (u_1, \pip^{a+b})_\p \, 
(\pip^a u_2, \pip^b u_3)_\p\ \\
&=  (u_1, \pip^{a+b})_\p \,  
\cancel{(u_2, u_3)_\p} \, (u_2, p^b)_\p \, 
(\pip^a, u_3)_\p \, (\pip^a, \pip^b)_\p \\
&=  (u_1, \pip)_\p^{ab} \,  
(u_2, \pip)_\p^b \, 
(u_3, \pip)_\p^a \, (\pip, \pip)_\p^{ab} \\
&=  \leg{u_1}{\p}^{a+b} 
\leg{u_2}{\p}^b 
\leg{u_3}{\p}^a  \leg{-1}{\p}^{ab}.
\end{align*}
\end{proof}

\begin{lem}[Local Mass Formula] \label{lemma:local_mass_formula}
Given a Jordan decomposition $Q \cong_{\Op} \oplus_{i=1}^r \pip^{\al_i} Q_i$ as  a sum of unimodular $\Op$-valued quadratic forms $Q_i$ where $n_i := \dim(Q_i)$ and the $\al_i$ are non-decreasing with $i$, we have the formula
$$
\beta_{Q, \p}(Q)^{-1} 
	= q^{2\cdot \ord_\p(\det_G(Q))} \cdot 
	2\prod_i M_\p(Q_i) \cdot \sum_{i < j} q^\frac{(\al_j - \al_i) \cdot n_i \cdot n_j}{2},
$$
where 
$$
M_\p(Q_i) := 
\begin{cases}
1 &\qquad \text{if $n_i=0$,} \\
\frac{1}{2} &\qquad \text{if $n_i=1$,} \\
\frac{1}{2\(1 -  \frac{\chi_u(\p)}{q}\)} &\qquad \text{if $n_i= 2$,} \\
\frac{1}{2\(1 - \frac{1}{q^2}\)} &\qquad \text{if $n_i= 3$,} \\
\end{cases}
$$
and $\chi_u(\p) := \leg{-u}{\p}$.
\end{lem}

\begin{proof}
This follows from \cite[eq (3), p263]{CS} and \cite[Remark 7.4, p28]{Ha-masses_of_varying_det} that these formulas also hold at all primes $\p\nmid 2$ of a number field $F$ by replacing the prime $p$ there by either $q$ or $\p$ as above.
\end{proof}


We now perform explicit local computations along the lines of \cite[\textsection 5]{Ha-masses_of_varying_det} to determine the contributions of each local genus to the quantities $A^*_{\p; n}(S)$ and $B^*_{\p;n}(S)$ when $n=3$.

\begin{thm}  \label{thm:four_Jordan_block_structures}
Suppose that $\p\nmid 2$.  Then there are exactly four Jordan block structures of size $n=3$, given by the tuples $(3), (2,1), (1,2), (1,1,1)$.  Their distributions of Hasse invariants and local density factors over all local genera $G_\p$ of primitive $\Op$-valued quadratic forms with Gram determinant squareclass $S$, as well as their contributions to $A^*_{\p; n=3}(S)$ and $B^*_{\p; n=3}(S)$, are given by the following two tables:
\end{thm}


\begin{table}[!h]
\caption{Local genus types and related invariants for primes $\p \nmid 2$.} 
\centering
\hspace{-0.2in} 
\resizebox{6in}{!}{
\begin{tabular}{  c |  c  |  c  |  c   | c  | c  | c | }
\multirow{2}{*}{\text{\#}}   & Valuation Type & Allowed & \multirow{2}{*}{Cases }& \multicolumn{2}{c|}{\# of Hasse Invariants}
    & Normalized Densities
\\
& $0 \leq a \leq b$ & $\nu = \ord_\p(S)$ & & $\hspace{.13in}c_\p = 1\hspace{.13in}$ & $c_\p = -1$ 
    & $(1- \frac{1}{q^2})\cdot\beta_{Q, \p}^{-1}(Q)$ 
\\
\hline 
1 & $0 = a = b$ & $\nu = 0$ & -- & 1 & 0 & 1\\
\hline
2a & \multirow{2}{*}{$0 = a < b$} & \multirow{2}{*}{$\nu \geq 1$} & $b \equiv 0_{(2)}$ 
    & 2 & 0 & \multirow{2}{*}{$\frac{1}{2q^b}\(1 + \frac{\ve' \cdot \ve_2}{q} \)$} 
    \\
2b & &  & $b \equiv 1_{(2)}$ 
    & 1  & 1 & 
    \\
 & &  &  
    & ($\ve_2 = 1$)  & ($\ve_2 = -1$) & 
    \\
\hline
3a & \multirow{2}{*}{$0 < a = b$} & \multirow{2}{*}{$\nu \geq 2,$ even} & $b \equiv 0_{(2)}$ 
    & 2 & 0 & \multirow{2}{*}{$\frac{1}{2q^{3b}}\(1 + \frac{\ve' \cdot \ve_3}{q} \)$} 
    \\
3b & &  & $b \equiv 1_{(2)}$ 
    & 1 & 1 & 
    \\
 & &  &  
    & ($\ve_3 = \ve'$)  & ($\ve_3 = -\ve'$) & 
    \\
\hline
4a & \multirow{2}{*}{$0 < a < b$} & \multirow{2}{*}{$\nu \geq 3$} & $a \equiv b \equiv 0_{(2)}$ 
    & 4 & 0 & \multirow{2}{*}{$\frac{1}{4q^{2a+b}}\(1 - \frac{1}{q^2} \)$ }
    \\
4b & &  & otherwise 
    & 2 & 2 & 
    \\
\hline
\end{tabular}
}
\label{Tab:odd_p_invariants}
\end{table}

\begin{table}[!h]
\caption{Local genus types and Euler factor information for primes $\p \nmid 2$.} 
\centering
\hspace{0.1in}
\begin{tabular}{  c |  c  |  c  |  c   | c  | c  | c | }
\multirow{2}{*}{\text{\#}}   & Valuation Type & Allowed & \multirow{2}{*}{Cases }
    & \multirow{2}{*}{$A^*_{\p}(S)$} & \multirow{2}{*}{$B^*_{\p}(S)$}    
\\
& $0 \leq a \leq b$ & $\nu = \ord_\p(S)$ & & & 
\\
\hline 
1 & $0 = a = b$ & $\nu = 0$ & -- & 1 & 1 \\
\hline
2a & \multirow{2}{*}{$0 = a < b$} & \multirow{2}{*}{$\nu \geq 1$} & $b \equiv 0_{(2)}$ 
    & \multirow{2}{*}{$\frac{1}{q^b}$} 
    &  $\frac{1}{q^b}$ 
    \\
2b & &  & $b \equiv 1_{(2)}$ 
    & 
    & $\ve' \cdot \frac{1}{q^{b+1}}$ 
    \\
\hline
3a & \multirow{2}{*}{$0 < a = b$} & \multirow{2}{*}{$\nu \geq 2,$ even} & $b \equiv 0_{(2)}$ 
    & \multirow{2}{*}{$\frac{1}{q^{3b}}$}
    & $\frac{1}{q^{3b}}$  
    \\
3b & &  & $b \equiv 1_{(2)}$ 
    & 
    & $ \frac{1}{q^{3b+1}}$  
    \\
\hline
4a & \multirow{2}{*}{$0 < a < b$} & \multirow{2}{*}{$\nu \geq 3$} & $a \equiv b \equiv 0_{(2)}$ 
    & \multirow{2}{*}{$\frac{1}{q^{2a+b}}\(1 - \frac{1}{q^2} \)$}
    & $\frac{1}{q^{2a+b}}\(1 - \frac{1}{q^2} \)$ 
    \\
4b & &  & otherwise 
    & 
    & 0
    \\
\hline
\end{tabular}
\label{Tab:odd_p_Euler_info}
\end{table}

\begin{proof}
In each case we use Lemma \ref{lemma:tuple} to enumerate  the local genera $G_\p$ with $\det_G(G_\p) = S$, and use Lemmas \ref{lemma:Hasse} and \ref{lemma:local_mass_formula} to compute their Hasse invariants and normalized local densities.
\end{proof}

\section{Explicit ternary Euler Factors for primes $\p\nmid2$}



In this section we use Table \ref{Tab:odd_p_Euler_info} above to explicitly compute the local Dirichlet series 
described in \cite[Thrm 5.4, pp22-3]{Ha-masses_of_varying_det} for all primes $\p \nmid 2$.

\begin{lem} \label{Lem:fixed_u_odd_A-series}
Suppose that $\p \nmid 2$ and fix some local unit squareclass 
$u \in \SqCl(\Op^\times)$.  Then
$$
\sum_{
	\substack{S_\p \in \SqFpInt \\ \text{with $\widetilde{S_\p}= 2u$}}
	}
	\frac{A^*_{\p; n=3}(S_{\p})}{N_{F/\Q}(\I(S_{\p}))^{s}}
=
\frac{1 - q^{-(3s+6)}}
{(1 - q^{-(s+1)})(1 - q^{-(2s+3)})}
$$
as an equality of formal Dirichlet series.
\end{lem}

\begin{proof}
We temporarily set $X:= q^{-s}$ for notational convenience, and compute
{
\tiny
\begin{align*}
\sum_{\nu \geq 0} A^*_{\p}(S_{\p^\nu}) X^{\nu} 
&= 
1 
+ 
    \sum_{\nu \geq 1} p^{-\nu} X^{\nu} 
+ 
    \sum_{\substack{\nu \geq 2 \\ \text{$\nu$ even}}} q^{-\frac{3\nu}{2}} X^{\nu}  
+ \sum_{\nu \geq 3} 
    \,\,\,
    \sum_{\substack{0 < a < b \\  \text{where} \\a + b = \nu}} 
    \(1 - \frac{1}{p^2} \) 
    q^{-(2a + b)} X^{\nu}  
\\
&= 
\cancel{1} + \( \frac{1}{1-q^{-1}X} \cancel{-1}\) + \(\frac{1}{1-q^{-3}X^2} - 1\)  
+ \(1 - \frac{1}{q^2} \) \sum_{\nu \geq 3} 
    \,\,\,
    \sum_{\substack{0 < a < b \\  \text{where} \\a + b = \nu}} 
    q^{-(2a + b)} X^{\nu}  
\\
&= 
-1 +  \frac{1}{1-q^{-1}X}  + \frac{1}{1-q^{-3}X^2}   
+ \(1 - \frac{1}{q^2} \) q^{-4} X^3 \sum_{\substack{\nu' \geq 0  \text{ with} \\ \nu':= \nu - 3}} 
    \,\,\,
    \sum_{\substack{0 \leq a' \leq b'  \text{ where} \\a' + b' = \nu'  \text{ and} \\ a' := a-1,\, b' := b-2}} 
    q^{-(2a' + b')} X^{\nu'}  
\end{align*}
}
Now using (weighted) Ferrer diagrams to read the restricted (weighted) partition of $\nu'$ into two parts as a (weighted) partition of $\nu'$ into parts of size at most two, we can rewrite the double sum as
\begin{equation} \label{Eq:Ferrer_double_sum}
\sum_{\nu' \geq 0} 
    \,\,\,
    \sum_{\substack{0 \leq a' \leq b'  \text{ where} \\a' + b' = \nu'}} 
        q^{-(2a' + b')} X^{\nu'}
= \frac{1}{1 - q^{-1}X} \cdot \frac{1}{1 - q^{-3}X^2}
\end{equation}
giving
\begin{align*}
\sum_{\nu \geq 0} A^*_{p^\nu} X^{\nu} 
&= 
-1 +  \frac{1}{1-q^{-1}X}  + \frac{1}{1-q^{-3}X^2}   
+ \frac{\(1 - \frac{1}{q^2}\) q^{-4} X^3}{(1 - q^{-1}X)(1 - q^{-3}X^2)}
\\
&= 
\frac{-(1 - q^{-1}X)(1 - q^{-3}X^2) + (1 - q^{-1}X) + (1 - q^{-3}X^2) + \(1 - \frac{1}{q^2}\) q^{-4} X^3}
{(1 - q^{-1}X)(1 - q^{-3}X^2)}
\\
&= 
\frac{1 - q^{-4}X^3 + \(1 - \frac{1}{q^2}\) q^{-4} X^3}
{(1 - q^{-1}X)(1 - q^{-3}X^2)}
\\
&= 
\frac{1 - q^{-6}X^3}
{(1 - q^{-1}X)(1 - q^{-3}X^2)}
.
\end{align*}
Finally, setting $X := q^{-s}$ gives the desired identity.
\end{proof}




\begin{lem}  \label{Lem:fixed_u_odd_B-series}
Suppose that $\p \nmid 2$ and we fix some local unit squareclass 
$u \in \SqCl(\Op^\times)$.  Then
$$
\sum_{
	\substack{S_\p \in \SqFpInt \\ \text{with $\widetilde{S_\p}= 2u$}}
	}
	\frac{B^*_{\p; n=3}(S_{\p})}{N_{F/\Q}(\I(S_{\p}))^{s}}
=
\frac{1 + \ve' \cdot q^{-(s+2)} + q^{-(2s+4)}}
	{1 - q^{-(2s+2)}}
$$
as an equality of formal Dirichlet series, where $\ve' := \leg{-1}{\p}$.
\end{lem}

\begin{proof}
We again temporarily define $X:= q^{-s}$ for notational convenience, and compute
{
\tiny
\begin{align*}
\sum_{\nu \geq 0} B^*_{\p^\nu} X^{\nu} 
&= 
1 
+ \(
    \sum_{\substack{\nu \geq 1 \\ \text{$\nu$ even}}} q^{-\nu} X^{\nu} 
         + \ve'  \sum_{\substack{\nu \geq 1 \\ \text{$\nu$ odd}}} q^{-\nu-1} X^{\nu} 
    \)
+ \(
    \sum_{\substack{\nu \geq 2 \\ \,\,\,\, \nu \equiv 0_{(4)}}} q^{-\frac{3\nu}{2}} X^{\nu}  
        + \sum_{\substack{\nu \geq 2 \\ \,\,\,\, \nu \equiv 2_{(4)}}} q^{-\frac{3\nu}{2}- 1} X^{\nu}  
    \)
    \\
& \qquad \qquad 
+ \sum_{\substack{\nu \geq 6 \\ \text{$\nu$ even}}} 
    \,\,\,
    \sum_{\substack{0 < a < b \\  \text{where} \\a + b = \nu \\ \text{$a, b$ even}}} 
    \(1 - \frac{1}{q^2} \) 
    q^{-(2a + b)} X^{\nu}  
\\
&= 
1 
+ \(
    \sum_{\substack{\nu_1:= \frac{\nu}{2} \geq 1}} q^{-2\nu_1} X^{2\nu_1} 
         + \ve'  \sum_{\substack{\nu_1 := \nu-1 \geq 0 \\ \text{$\nu_1$ even}}} q^{-\nu_1-2} X^{\nu_1 + 1} 
    \)
    \\
& \qquad \qquad 
+ \(
    \sum_{\substack{\nu_1 := \frac{\nu}{4}\geq 1 }} q^{-6\nu_1} X^{4\nu_1}  
        + \sum_{\substack{\nu_1 := \frac{\nu}{2} \geq 1 \\  \text{$\nu_1$ odd}}}
             q^{-3\nu_1 - 1} X^{2\nu_1}  
    \)
    \\
& \qquad \qquad 
+ \(1 - \frac{1}{p^2} \)\sum_{\substack{\nu_1\geq 0 \\ \text{$\nu_1$ even} \\ \text{with} \\ \nu_1:= \nu - 6}} 
    \,\,\,
    \sum_{\substack{0 \leq a_1 \leq b_1  \text{ where} \\a_1 + b_1 = \nu_1 \\ \text{$a_1, b_1$ even, with}  \\ a_1 := a-2,\, b_1 := b-4}} 
    q^{-(2a_1 + b_1) -8} X^{\nu_1 + 6} 
\\
&= 
1 
+ \( \[\frac{1}{1 - q^{-2}X^2} - 1\] + \ve' \cdot \frac{q^{-2}X}{1 - q^{-2}X^2} \)
+ \( \[\frac{1}{1 - q^{-6}X^4} - 1\] + \frac{q^{-4}X^2}{1 - q^{-6}X^4} \)
\\
& \qquad \qquad 
+ \(1 - \frac{1}{q^2} \) q^{-8} X^6 \sum_{\substack{\nu' \geq 0 \\ \text{$\nu'$ even} \\ \text{with} \\ \nu':= \nu - 6}} 
    \,\,\,
    \sum_{\substack{0 \leq a' \leq b'  \text{ where} \\a' + b' = \nu' \\ \text{$a', b'$ even, with}  \\ a' := a-2,\, b' := b-4}} 
    q^{-(2a' + b')} X^{\nu'} 
\end{align*}
}
\noindent
where the last double sum can be evaluated as $\[ (1-q^{-2}X^2) (1-q^{-6}X^4)\]^{-1}$ by taking the even exponent terms of Ferrer sum identity 
(\ref{Eq:Ferrer_double_sum}).  This then can be written as the rational function
{
\tiny
\begin{align*}
\sum_{\nu \geq 0} B_{\p^\nu} X^{\nu}
&= 
\frac{1 + \ve' \cdot q^{-2}X}{1 - q^{-2}X^2}
+ \frac{q^{-4}X^2 + q^{-6}X^4}{1-q^{-6}X^4}
+ \frac{\(1 - \frac{1}{q^2}\) q^{-8} X^6}{(1 - q^{-2}X^2)(1 - q^{-6}X^4)}
\\
&= 
\frac
{
(1 + \ve' \cdot q^{-2}X)(1 - q^{-6}X^4) 
+ (q^{-4}X^2 + q^{-6}X^4) (1 - q^{-2}X^2)
 + (1 - q^{-2}) q^{-8} X^6
}
{(1 - q^{-2}X^2)(1 - q^{-6}X^4)}
\\
&= 
\frac{
(1 +  \ve' \cdot q^{-2}X \cancel{- q^{-6}X^4} -  \ve' \cdot q^{-8}X^5)
+ (\cancel{q^{-6}X^4} \cancel{- q^{-8}X^6} + q^{-4}X^2 - q^{-6}X^4)
+ (\cancel{q^{-8} X^6} - q^{-10} X^6)
}
{(1 - q^{-2}X^2)(1 - q^{-6}X^4)}
\\
&= 
\frac{
1 +  \ve' \cdot q^{-2}X + q^{-4}X^2
- q^{-6}X^4 -  \ve' \cdot q^{-8}X^5
- q^{-10} X^6
}
{(1 - q^{-2}X^2)(1 - q^{-6}X^4)} 
\\
&= 
\frac{
\cancel{(1 - q^{-6}X^4)}(1 +  \ve' \cdot q^{-2}X + q^{-4}X^2)
}
{(1 - q^{-2}X^2) \cancel{(1 - q^{-6}X^4)}}
\end{align*}
}
\end{proof}


\begin{lem}[Computing Euler Factors] \label{Lem:Explicit_Odd_Euler_factors}
Suppose that $\p \nmid 2$, and 
define $S_{\p^\nu} := S(\p)^\nu$ for any fixed choice of $S(\p) \in \SqFpInt$ with $\I(S(\p)) = \p$.
Then we have 
$$
\sum_{\nu=0}^\infty
	\frac{A^*_{\p; n=3}(S_{\p^\nu})}{N_{F/\Q}(\I(S_{\p^\nu}))^{s}}
= \frac{1 - q^{-(3s+6)}}
{(1 - q^{-(s+1)})(1 - q^{-(2s+3)})}
$$
and 
$$
\sum_{\nu=0}^\infty
	\frac{B^*_{\p; n=3}(S_{\p^\nu})}{N_{F/\Q}(\I(S_{\p^\nu}))^{s}}
= \frac{1 + \ve' \cdot q^{-(s+2)} + q^{-(2s+4)}}
{1 - q^{-(2s+2)}}
$$
as an equality of formal Dirichlet series, where $\ve' := \leg{-1}{\p}$.
In particular, these are the Euler factors at primes $\p\nmid 2$ of the Dirichlet series $D_{A^*,\lambda; n=3}(s)$ and $D_{B^*,\lambda; n=3}(s)$ associated to a formal squareclass series (in the sense of \cite[Defn 4.12, p20]{Ha-masses_of_varying_det}) when $n=3$.
\end{lem}

\begin{proof}
This follows from Lemmas \ref{Lem:fixed_u_odd_A-series} and \ref{Lem:fixed_u_odd_B-series} since the local Dirichlet series having fixed normalized squareclass $\widetilde{S}_\p$ are independent of the choice of $\widetilde{S}_\p$.
\end{proof}




\section{Local ternary Computations at primes $\p \mid 2$ with $F_\p = \Q_2$}

In this section we enumerate local genera of primitive $\Op$-valued quadratic forms at primes $\p\mid2$ when $F_\p = \Q_2$ via the Conway-Sloane canonical 2-adic genus symbols described in \cite[Chapter 15, Section 7.6, p382]{CS-book}.

\begin{defn}
Recall from \cite[Defn 5.1 and Remark 5.2, p22]{Ha-masses_of_varying_det} that a {\bf partial local genus symbol} of size $n$ is a tuple $(n_1, \cdots, n_r)$ of $n_i \in \N$ where $n_1 + \cdots + n_r = n$ and $r \in \N$ is not specified, decorated in a specified way by various separators and overbars.  Also these symbols are in bijection with 
local genus symbols over $\Q_2$ not having specified choices of signs or oddities.
\end{defn}

\begin{lem} \label{Lem:twenty_plgs}
There are exactly 20 partial local genus symbols of size $n=3$,
 given by: \\
 \indent
 $(3), \quad (2.1), (2;1), (2::1),  (\bar{2}, 1), (\bar{2} :: 1),  \quad (1, 2), (1;2), (1 :: 2), (1, \bar{2}), (1::\bar{2}),$ 
 \\
\indent
 $(1,1,1), (1,1;1), (1;1,1), (1;1;1), (1,1::1), (1::1,1), (1::1;1), (1;1::1), (1::1::1)$.  
\end{lem}

\begin{proof}
These can be enumerated by decorating the four Jordan block structures of size 3 in Theorem \ref{thm:four_Jordan_block_structures} with separators and bars as allowed by the definition.
\end{proof}

\begin{rem}[Type II${}_2$ diagonalization]  \label{Rem:type_II_diagonalizations}
For computing the Hasse invariants of genera containing a type II${}_2$ block 
(i.e. genera with partial local genus symbol having a $\bar{2}$ entry) it is useful to have a $\Q_2$-diagonal form readily available.  The $\Q_2$-diagonalizations of the two unimodular type II${}_2$ quadratic forms over $\Z_2$ are given by
$$
Q_1 := \begin{bmatrix}
0 & 1 \\ 1 & 0 
\end{bmatrix}
\sim_{\Q_2} 
\begin{bmatrix}
-2 & 0 \\ 0 & 2 
\end{bmatrix},
\qquad {\textstyle\det_G(Q_1)} \equiv 7_{\pmod 8} \implies \text{sign} = 1,
$$
$$
Q_2 := \begin{bmatrix}
2 & 1 \\ 1 & 2 
\end{bmatrix}
\sim_{\Q_2} 
\begin{bmatrix}
2 & 0 \\ 0 & 6 
\end{bmatrix},
\qquad {\textstyle\det_G(Q_2)} \equiv 3_{\pmod 8} \implies \text{sign} = -1.
$$
\end{rem}


As a first step towards computing the local densities $\beta_{Q, 2}(Q)$ for each partial local genus listed in Lemma \ref{Lem:twenty_plgs}, and we would like to know how the local densities depend on the choices of signs and oddities that we use to decorate a partial local genus symbol.

\begin{lem}[Local Density dependence]  \label{Lem:Local_density_dependence}
The genera of primitive $\Z_2$-valued ternary quadratic forms with partial local genus symbol $\P$ having a `,'
separator
(i.e. Cases \#2, 5, 7, 10, 12--20)  
have local density $\beta_{Q, 2}(Q)$ depending only on $\P$.
For other genera, the dependence of $\beta_{Q, 2}(Q)$ on the genus symbol is given by 
$$
\begin{tabular}{c | c}
$\P$ & $\beta_{Q, 2}(Q)$ depends on \\
\hline
(3) & the octane of the 3-dim'l Jordan block (I${}_3$-octane) \\
(1,2), (1;2), (2,1), (2;1) & the sign of the 2-dim'l Jordan block (I${}_2$-oddity) \\
($\bar{2}$ :: 1), (1 :: $\bar{2}$) & the sign of the train with the 2-dim'l Jordan block (II${}_2$-sign)
\end{tabular}
$$
and their associated normalized local densities $\(1 - \frac{1}{2^2}\) \beta_{Q, 2}(Q)^{-1}$ are listed explicitly in Table 5.2.
\end{lem}

\begin{proof}
The local density $\beta_{Q, 2}(Q)$ is related to the local masses $m_2(Q)$ defined on \cite[eqn (3), p8]{CS} by 
a factor that depends only on the partial local genus symbol $\P$.  The local mass $m_2(Q)$ is a product of three factors, though only one of them (the diagonal factor) depends on the sign and oddities not specified in 
$\P$.  
The diagonal factors are determined by the species labels associated to each of the Jordan blocks, and the species of a bound block is determined by its partial local genus symbol, so we only need to consider free blocks.  These species labels are determined by the quantities in the table above. 
\end{proof}




We would also like 
to compute the distribution of Hasse invariants $c_2$ across local genera obtained from  each partial local genus symbol of specified local determinant squareclass $S_\p \in \SqFpInt$.

\begin{defn}
%
Given a quadratic form $Q$ over $F_\p = \Q_2$ with (Hessian or Gram) determinant squareclass $S_\p \in \SqFpInt$, we define its {\bf normalized unit determinant squareclass} as the squareclass 
$\frac{1}{4}\widetilde{S_\p} \in \SqCl(\Op^\times)$, and this is determined by its 
reduction (mod $8\Op$) since  $\SqCl(\Op^\times) \cong(\Z/8\Z)^\times$.  Also the normalized unit determinant squareclass is the same if we use either the Hessian or Gram determinants, since they differ by a power of $\pip$, because our convention takes $\pip=2$.
\end{defn}

\smallskip
One useful observation from normalizing unit squareclasses is that we can understand the contribution of $\p$-power scalings to the Hasse invariant separately from those only involving unit squareclasses.

\begin{lem}[Hasse invariant structure] \label{Lem:Hasse_invariant_structure}
Suppose that $F_\p= \Q_2$.  Then every (non-degenerate) primitive $\Op$-valued ternary quadratic form $Q$ is equivalent over $\Q_2$ to a diagonal form $ax^2 + 2^\beta b y^2 + 2^\gamma c z^2$  where 
$a, b, c \in \{1, 3, 5, 7\}$ 
and 
$\beta, \gamma \in \Z$.  The Hasse invariant of $Q$ is given by
$$
c_\p  := c_\p(Q)  = c_2(Q) = \underbrace{(a,b)_2 (a,c)_2 (b,c)_2}_{c_{2, unit} :=} \cdot \hspace{-.08in}
	\underbrace{(2^\beta, ac)_2 (2^\gamma, ab)_2}_{\text{valuation adjustment $:=$}},
$$
 and here $c_{2, unit}$ depends only $a, b, c \pmod {4\Op}$.
\end{lem}
\begin{proof}
This follows from the definition of the Hasse invariant of a diagonal quadratic form (see \cite[p55]{Ca}), and the fact that the 2-adic Hilbert symbol $(\cdot, \cdot)_2$ restricted to unit squareclasses $u(\Q_2^\times)^2$ with $u \in \Z_2^\times$ depends only on their reductions mod 8, and that  $(2,2)_2 =1$. (See e.g. \cite[p43]{Ca}.)
\end{proof}

\begin{defn}
We refer to $c_{2, unit}$ in Lemma \ref{Lem:Hasse_invariant_structure} as the  {\bf unit Hasse invariant} and call the Hilbert symbol product $(2^\beta, ac)_2 (2^\gamma, ab)_2$ above the {\bf valuation adjustment} to the unit Hasse invariant.  Since we can collect all Hilbert symbols $(2, \cdot)_2$ in the valuation adjustment there, we see that it can be written as either $(2, ab)_2, (2, bc)_2$ or $(2, ac)_2$, and so it depends on at most two of the units $a,b,c$.
We say that a valuation adjustment {\bf depends on a train} if some unit appearing in the valuation adjustment $(2, \cdot)_2$ is on that train.  From its definition, we can see that a valuation adjustment depends on at most two trains.  
\end{defn}

\smallskip
\begin{center}
\line(1,0){250}
\end{center}
\smallskip

Our strategy for computing the $c_2$-distributions is to consider all ways of decorating a partial local genus symbol to  give a local genus symbol with the specified normalized unit determinant squareclass.  In terms of local integral invariants, this ``decoration'' happens by assigning oddities to each compartment and then by associating a sign to each train.

\begin{defn}
Given a local genus symbol $\G$, we have a natural map $\varphi: \G \mapsto \P$ where $\P$ is the partial genus symbol associated to $\G$ by forgetting about signs and oddities, and translating the remaining train/compartment structure into $\P$ by using \cite[Rem 5.2, p22]{Ha-masses_of_varying_det}.
Given $\P$, we say that any $\G$ with $\varphi(\G) = \P$ is a {\bf decoration} of $\P$.
We also have a distinguished subset of the decorations $\G \in \varphi^{-1}(\P)$ of $\P$ whose signs on each train are $+$, and we refer to these as the {\bf unsigned decorations} of $\P$ or as the {\bf unsigned local genus symbols} associated to $\P$.
\end{defn}

The next lemma shows that the sign decorations can be done independently of the oddity decorations, and that the effect of the sign decorations  are easy to understand.  This will allow us to reduce to first computing the $c_2$-distributions of these oddity decorations alone.

\begin{lem}[Simplifying sign decorations]  \label{Lem:Simplifying_sign_decorations}
For a given partial local genus symbol $\mathcal{P}$ at $p=2$, we consider all local genus symbols $\mathcal{G}$ of fixed normalized unit determinant $u$ that arise from decorating $\mathcal{P}$ with oddities and signs.  Then the oddity decorations of $\mathcal{P}$ depend only on the reduction $u'$ of $u \pmod 4$, and the overall sign (defined as the product of the signs of all trains) determines which lift of $u'$ in $(\Z/8\Z)^\times$ is obtained. Thus specifying the normalized unit determinant squareclass $u$ (with given mod 4 reduction $u'$) is equivalent to specifying the overall sign of $\mathcal{G}$.
\end{lem}

\begin{proof}
This follows since the value of the (extended) Kronecker symbol
$$
\leg{u}{2} := 
\begin{cases}
+1 & \qquad \text{if $u \equiv \pm1 \pmod 8$,}\\
-1 & \qquad \text{if $u \equiv \pm3 \pmod 8$,}\\
\end{cases}
$$
exactly determines a unique $u \in (\Z/8\Z)^\times$ lifting $u' := u \pmod 4$, and the overall sign is the homomorphism from the normalized unit determinants in $(\Z/8\Z)^\times$ to $\{\pm1$\}  given by $u \mapsto \leg{u}{2}$.
\end{proof}

\smallskip
The next lemma tells us how the unit Hasse invariant is distributed as we vary over local genus symbols of fixed normalized unit determinant squareclass $u'$, unimodular block stucture and sign (taken to be $+$ in every train).  
This amounts to describing the $c_{2, unit}$ distribution for a partial local genus symbol $\mathcal P$ as we decorate it with oddities and fix its unit determinant $u'$ mod 4.

\begin{lem}[Fixed sign $c_2$-distributions] \label{Lem:c_2_unit_distributions}
Suppose that $F_\p = \Q_2$ and we are given $u' \in \SqCl(\Op^\times)$.  
Then the distribution of the unit Hasse invariants $c_{2, unit}$ across all local genera of (non-degenerate) $\Op$-valued ternary quadratic forms with normalized unit determinant squareclasses $u\equiv u'$ mod 4 is given by:

\medskip
\centerline{
\begin{tabular}{  c c || c   c  || c  c || c  c || c  c | }
 \multicolumn{2}{c||}{Multiplicity} & \multicolumn{2}{c||}{$I_1 \oplus I_1 \oplus I_1$}
     & \multicolumn{2}{c||}{$I_2 \oplus I_1$} & \multicolumn{2}{c||}{One $I_3$}
     & \multicolumn{2}{c|}{$II_2 \oplus I_1$} \\
 \multicolumn{2}{c||}{of $c_{2, unit}$ values} & \multicolumn{2}{c||}{Jordan blocks}
     & \multicolumn{2}{c||}{Jordan blocks} & \multicolumn{2}{c||}{Jordan block}
     & \multicolumn{2}{c|}{Jordan blocks}\\
\hline
 \multicolumn{2}{c||}{$c_{2, unit} = $} & \,\,$+1$ & $-1$ & \,\,$+1$ & $-1$ & \,\,$+1$ & $-1$ & \,\,$+1$ & $-1$ \\
\hline
\hline
\multirow{2}{*}{\text{Det}}  
& $u' \equiv 1_{(4)}$ & \,\,\,\,\,\,1 &  3 &  \,\,\,\,\,\,1 &  2  &  \,\,\,\,\,\,1 & 1 & \,\,\,\,\,\, 0 & 1 \\
& $u' \equiv -1_{(4)}$ & \,\,\,\,\,\,3 & 1 &  \,\,\,\,\,\,2 & 1  &  \,\,\,\,\,\,1 & 1  & \,\,\,\,\,\,1 & 0\\
\end{tabular}
}
\end{lem}

\begin{proof}
In preparation for the proof, we first enumerate the possible (unimodular) Jordan block representatives that can appear in a ternary quadratic form (up to integral equivalence), and describe their distributions of $c_{2, unit}$ invariants.  We will use these later to compute the distributions of Hasse invariants $c_2$ as we vary over all ways of decorating a partial local genus symbol of given determinant. 

\medskip
{\bf Block I${}_3$ Tables:}
$$
\begin{tabular}{  c c | c  c  c  c | }
\multicolumn{2}{c |}{Representative}  & \multicolumn{4}{c|}{\text{Oddity}}\\
\multicolumn{2}{c |}{Quadratic Forms}  & 1 & 3 & 5 & 7 \\
\hline
\multirow{2}{*}{\text{Sign}}  
& + & diag[1, 1, -1] & diag[1, 1, 1] & diag[-1, -1, -1] & diag[1, -1, -1] \\
& -- & diag[-1, -1, 3] & diag[1, -1, 3] & diag[1, 1, 3] & diag[1, 1, -3] \\
\end{tabular}
$$
$$
\begin{tabular}{  c c | c  c  c  c | }
 \multicolumn{2}{c|}{Values} & \multicolumn{4}{c|}{\text{Oddity}}\\
 \multicolumn{2}{c|}{of $c_{2, unit}$} & 1 & 3 & 5 & 7 \\
\hline
\multirow{2}{*}{\text{Sign}}  
& + & $+1$ &  $+1$ &  $-1$ &  $-1$ \\
& -- & $-1$ & $-1$ & $+1$ & $+1$ \\
\end{tabular}
\qquad\qquad
\begin{tabular}{  c c | c  c  c  c | }
 \multicolumn{2}{c|}{Values} & \multicolumn{4}{c|}{\text{Oddity}}\\
 \multicolumn{2}{c|}{of $\det$} & 1 & 3 & 5 & 7 \\
\hline
\multirow{2}{*}{\text{Sign}}  
& + & $7_{(8)}$ &  $1_{(8)}$ &  $7_{(8)}$ &  $1_{(8)}$ \\
& -- & $3_{(8)}$ & $5_{(8)}$ & $3_{(8)}$ & $5_{(8)}$ \\
\end{tabular}
$$
So here we see that the $c_{2, unit}$ values are equidistributed on the determinant level sets.

\medskip
{\bf Block I${}_2$ Tables:}
$$
\begin{tabular}{  c c | c  c | c  c | }
\multicolumn{2}{c |}{Representative} & \multicolumn{4}{c|}{\text{Oddity}}\\
\multicolumn{2}{c |}{Quadratic Forms} & 2 & -2 & 0 & 4 \\
\hline
\multirow{2}{*}{\text{Sign}}  
& + & diag[1, 1] & diag[-1, -1] & diag[1, -1] & -  \\
& -- & diag[-1, 3] & diag[1, -3] & - & diag[-1, -3] \\
\hline
& & \multicolumn{2}{c}{$\det = 1_{(4)}$} & \multicolumn{2}{|c|}{$\det = -1_{(4)}$}\\ 
\end{tabular}
$$

$$
\begin{tabular}{  c c | c  c  c  c | }
\multicolumn{2}{c |}{Values of} & \multicolumn{4}{c|}{\text{Oddity}}\\
\multicolumn{2}{c | }{$(a, b)_2$}  & 2 & -2 & 0 & 4 \\
\hline
\multirow{2}{*}{\text{Sign}}  
& + & $+1$ & $-1$ & $+1$ & -  \\
& -- & $-1$ & $+1$ & - & $+1$ \\
\hline
& & \multicolumn{2}{c|}{$\det = 1_{(4)}$} & \multicolumn{2}{c|}{$\det = -1_{(4)}$}\\ 
\end{tabular}
\qquad \qquad
\begin{tabular}{  c c | c  c  c  c | }
\multicolumn{2}{c |}{Values} & \multicolumn{4}{c|}{\text{Oddity}}\\
\multicolumn{2}{c | }{of $\det$}  & 2 & -2 & 0 & 4 \\
\hline
\multirow{2}{*}{\text{Sign}}  
& + & $1_{(8)}$ & $1_{(8)}$ & $7_{(8)}$ & -  \\
& -- & $5_{(8)}$ & $5_{(8)}$ & - & $3_{(8)}$ \\
\hline
& & \multicolumn{2}{c|}{$\det = 1_{(4)}$} & \multicolumn{2}{c|}{$\det = -1_{(4)}$}\\ 
\end{tabular}
$$

\medskip
{\bf Block I${}_1$ Table:}
$$
\begin{tabular}{  c c | c  c | c  c | }
\multicolumn{2}{c |}{Representative} & \multicolumn{4}{c|}{\text{Oddity $=$ Determinant}}\\
\multicolumn{2}{c |}{Quadratic Forms} & 1 & 5 & 3 & 7 \\
\hline
\multirow{2}{*}{\text{Sign}}  
& + & diag[1] & - & - &  diag[-1]  \\
& -- & - & diag[-3] & diag[3] & - \\
\hline
& & \multicolumn{2}{c}{$\det = 1_{(4)}$} & \multicolumn{2}{|c|}{$\det = -1_{(4)}$}\\ 
\end{tabular}
$$
For Type $\text{I}_1$ blocks we always have $c_2 = c_{2, unit} = 1$.

\medskip
{\bf Block II${}_2$ Table:}
$$
\begin{tabular}{  c c | c  c | c  c | }
\multicolumn{2}{c |}{Representative} & \multicolumn{4}{c|}{\text{Determinant}}\\
\multicolumn{2}{c |}{Quadratic Forms} & 1 & 5 & 3 & 7 \\
\hline
\multirow{2}{*}{\text{Sign}}  
& + & - & - & - &  $\begin{bmatrix} 0 & 1 \\ 1 & 0 \end{bmatrix}$  \\
& -- & - & - & $\begin{bmatrix} 2 & 1 \\ 1 & 2 \end{bmatrix}$ & - \\
\hline
& & \multicolumn{2}{c}{$\det = 1_{(4)}$} & \multicolumn{2}{|c|}{$\det = -1_{(4)}$}\\ 
\end{tabular}
$$
For Type II blocks the oddity is defined to be zero, and from the diagonal forms of these block over $\Q_2$ we can easily check that $c_2 = c_{2, unit} = -1$ for both of these blocks.


We now compute the table of multiplicities of $c_{2, unit}$ values appearing in the table of Jordan blocks assuming that all signs are $+$ (and so the overall unit determinant is $\pm1 \pmod 8$).  The four ways of writing a ternary form as a sum of scaled unimodular types have unimodular types  I${}_3$, I${}_2 \oplus$ I${}_1$, I${}_1 \oplus$ I${}_1 \oplus$ I${}_1$ and II${}_2 \oplus$ I${}_1$.

\medskip
{\bf Case I${}_3$:}  This part of the table follows directly from counting the number of  $c_2$ invariants (by value) of type $I_3$ blocks of fixed determinant $\pm1 \pmod 8$.

\medskip
{\bf Case I${}_2 \oplus$ I${}_1$:}  Here we compute the Hasse invariant $c_2$ for the direct sums of each of the possible oddities of each summand (both with $sign=+$) by the usual formula \cite[Lemma 2.3(iii), p58]{Ca}.  
\medskip
$$
\begin{tabular}{  c c | c  c | c  | }
\multicolumn{2}{c |}{$I_2 \oplus I_1$ } & \multicolumn{3}{c|}{\text{$I_2$ oddities}}\\
\multicolumn{2}{c |}{$c_{2, unit}$} & 2 & -2 & 0  \\
\hline
\multirow{2}{*}{\text{$I_1$ oddities}}  
& 1 & +1 & -1 & +1  \\
& 7 & +1 & -1 & -1  \\
\hline
\end{tabular}
\qquad \qquad
\begin{tabular}{  c c | c  c | c  | }
\multicolumn{2}{c |}{$I_2 \oplus I_1$ } & \multicolumn{3}{c|}{\text{$I_2$ oddities}}\\
\multicolumn{2}{c |}{$\det$} & 2 & -2 & 0  \\
\hline
\multirow{2}{*}{\text{$I_1$ oddities}}  
& 1 & $1_{(8)}$ & $1_{(8)}$ & $7_{(8)}$  \\
& 7 & $7_{(8)}$ & $7_{(8)}$ & $1_{(8)}$  \\
\hline
\end{tabular}$$

\medskip
{\bf Case I${}_1 \oplus$ I${}_1 \oplus$ I${}_1$:}  Here we compute the Hasse invariant $c_2$ for the direct sums of each of the possible oddities of each summand (both with $sign=+$) by the usual formula \cite[Lemma 2.3(iii), p58]{Ca}, in two steps.  
\medskip
$$
\begin{tabular}{  c c | c  c |  }
\multicolumn{2}{c |}{$I_1 \oplus I_1$ } & \multicolumn{2}{c|}{\text{$I_1$ oddities}}\\
\multicolumn{2}{c |}{$c_{2, unit}$} & 1 & 7   \\
\hline
\multirow{2}{*}{\text{$I_1$ oddities}}  
& 1 & +1 & +1   \\
& 7 & +1 & -1   \\
\hline
\end{tabular}
\qquad \qquad
\begin{tabular}{  c c | c  c | }
\multicolumn{2}{c |}{$I_1 \oplus I_1$ } & \multicolumn{2}{c|}{\text{$I_1$ oddities}}\\
\multicolumn{2}{c |}{$\det$} & 1 & 7   \\
\hline
\multirow{2}{*}{\text{$I_1$ oddities}}  
& 1 & $1_{(8)}$ & $7_{(8)}$  \\
& 7 & $7_{(8)}$ & $1_{(8)}$  \\
\hline
\end{tabular}$$

\medskip
$$
\begin{tabular}{  c c | c  c | c  c | }
\multicolumn{2}{c |}{$(I_1 \oplus I_1) \oplus I_1$ } & \multicolumn{4}{c|}{\text{$I_1$ oddities}}\\
\multicolumn{2}{c |}{$c_{2, unit}$} & \multicolumn{2}{c|}{1} & \multicolumn{2}{c|}{7}   \\
\hline
\multirow{2}{*}{\text{$I_1 \oplus I_1$ as above}}  
&  & +1 & +1 & +1 & -1 \\
&  & +1 & -1  & -1 & -1 \\
\hline
\end{tabular}
\qquad \qquad
\begin{tabular}{  c c | c  c | c  c | }
\multicolumn{2}{c |}{$(I_1 \oplus I_1) \oplus I_1$ } & \multicolumn{4}{c|}{\text{$I_1$ oddities}}\\
\multicolumn{2}{c |}{$\det$} & \multicolumn{2}{c|}{1} & \multicolumn{2}{c|}{7}   \\
\hline
\multirow{2}{*}{\text{$I_1 \oplus I_1$ as above}}  
&  & $1_{(8)}$ & $7_{(8)}$  & $7_{(8)}$ & $1_{(8)}$ \\
&  & $7_{(8)}$ & $1_{(8)}$  & $1_{(8)}$ & $7_{(8)}$ \\
\hline
\end{tabular}
$$

\medskip
{\bf Case II${}_2 \oplus$ I${}_1$:}  Here we again compute the Hasse invariant $c_{2, unit}$ for the direct sums of each of the possible oddities of each summand (both with $sign=+$) by the usual formula \cite[Lemma 2.3(iii), p58]{Ca}.  
There are only two forms $diag[-2,2,1]$ and $diag[-2,2,-1]$ in this case,  which have $c_{2,unit} = 1$ and $c_{2,unit} = -1$ respectively.
\end{proof}

At this point we need only describe how the $c_2$-distributions change as we further decorate our unsigned local genus symbol with signs, which determines the contribution of the valuation adjustment.

\begin{lem}[Adding sign decorations]  \label{Lem:Adding_sign_decorations}
The distribution of the Hasse invariants $c_{2}$ can be computed from the $c_{2, unit}$ distribution by the following rules:
\begin{enumerate}
\item Determine the appropriate overall sign $sign(u)$ for the of genus symbol, 

\item If the overall sign is $-1$, take the opposite $c_{2, unit }$-distribution from Lemma \ref{Lem:c_2_unit_distributions} if the valuation adjustment depends on all of the trains, otherwise take the original $c_{2, unit }$-distribution,

\item Replace this distribution by its average (constant) distribution if the valuation adjustment doesn't depend on all trains.

\item Scale this distribution by $2^{(\text{\# of trains}) - 1}$ to obtain the desired $c_2$-distribution.
\end{enumerate}
\end{lem}

\begin{proof}
In Step 1 we know that $sign(u)$ is the same as the sign of the overall sign of the genus symbol by definition of the overall sign.  

In Step 2, if the valuation adjustment depends on all trains then any change of sign (in any train) will change the  valuation adjustment, switching the $c_{2, unit}$ distribution.  Otherwise, we can change the sign in a train that does not affect the valuation adjustment, preserving the $c_{2, unit}$ distribution.  We also can check that if all signs of all Jordan blocks are $+1$ then the valuation adjustment is $+1$ because it is the product of the signs of the trains on which it depends.

In Step 3, we notice that further sign changes must occur in pairs to preserve the sign, and that there will be a pair switching the valuation adjustment iff there is a train that the valuation adjustment does not depend on (because we can choose exactly one sign change that the valuation adjustment depends on).

Step 4 gives the total number of sign changes we can get by switching pairs of signs, which is all possible ways of obtaining inequivalent local genus symbols, since the overall sign is fixed.
\end{proof}

\begin{thm}
Suppose that $F_\p = \Q_2$ and that $S_\p \in \SqFpInt$ in given.  Then the distribution of Hasse invariants $c_2(Q)$, normalized local densities $(1- \frac{1}{2^2})\cdot\beta_{2, Q}^{-1}(Q)$, and contributions to $A^*_\p(S_\p)$ and $B^*_\p(S_\p)$ for the local genera $\G$ with $\det_G(\G) = S_\p$ decorating each of the 20 partial local genus symbols $\P$ of Lemma \ref{Lem:twenty_plgs} are given in Tables 5.1--5.3.
\end{thm}

\begin{proof}
For each partial local genus symbol $\P$ we consider all decorations $\G$ of $\P$ with specified $\p$-power scalings for each Jordan block.  We use Remark \ref{Rem:type_II_diagonalizations} with Lemmas \ref{Lem:Hasse_invariant_structure} and \ref{Lem:Simplifying_sign_decorations}--\ref{Lem:Adding_sign_decorations} to compute the distribution of Hasse invariants over these local genera $\G$.  Then we use \cite[eq(2), p263]{CS} to compute the normalized local densities $(1- \frac{1}{2^2})\cdot\beta_{2, \G}^{-1}(\G)$ for these local genera.  Finally, to compute the contributions of these $\G$ to $A^*_\p(S_\p)$ and $B^*_\p(S_\p)$
we apply Lemma \ref{Lem:Local_density_dependence} to correlate the 
$c_2$-distributions and local densities already computed.
\end{proof}

\vfil

{\bf Editorial  Note:} The following page is left blank due to the table formatting in the following section.




\section{Tables for ternary $\p\mid2$ local computations when $F_\p = \Q_2$} \label{Sec:p=2_tables}

\TableHasseOne

\TableHasseTwo

\TableMassOne

\TableMassTwo

\TableEulerOne

\TableEulerTwo






\section{The ternary Euler factors for $A^*_{\p}$ and $B^*_{\p}$ for $\p\mid 2$ when $F_\p= \Q_2$} \label{Section:p_even}

In this section we assemble the computations in the previous section of $A^*_{\p; n=3}(S)$ and $B^*_{\p; n=3}(S)$ when $\p\mid 2$ and $F_\p= \Q_2$ to explicitly compute the formal local Dirichlet series described in \cite[Thrm 5.4, pp22-23]{Ha-masses_of_varying_det}.


\begin{lem} \label{Lem:A_u_even_prime_formula}
Suppose that $\p \mid 2$, $F_\p = \Q_2$, and we fix some 
$u \in \SqCl(\Op^\times)$.  Then
$$
\sum_{
	\substack{S_\p \in \SqFpInt \\ \text{with $\widetilde{S_\p}= 2u$}}
	}
	\frac{A^*_{\p; n=3}(S_{\p})}{N_{F/\Q}(\I(S_{\p}))^{s}}
= \frac{(1 - 2^{-(3s+6)}) \cdot 2^{-s}}
{(1-2^{-(s+1)})(1-2^{-(2s+3)})}
$$
as an equality of formal Dirichlet series.
\end{lem}

\begin{proof}
Since the generic local density at $\p$ is $\beta_{n=3, \p}(2)^{-1} = \frac{4}{1 - \tfrac{1}{2^2}}$, from Table 5.3 we see that the quantities $A^*_{\p}(S_{\p})$ are independent of $u \in \SqCl(\Op^\times) \cong (\Z/8\Z)^\times$, so 
the series  
$4 \cdot \sum_\nu
A^*_{\p}(S_\p)X^{\nu-3}$ with $X := q^{-s}$ and $\nu := \ord_\p(\I(S_\p))$ 
for the Gram determinant (having valuation $\nu-3$)
is given by
\begin{align*}
4 \cdot \sum_{\nu=0}^\infty A^*_{\p}(S_\p)X^{\nu-3}
& = 
2^{-1} 
+ 2^{-4} \cdot 3 \cdot X 
+ 2^{-5} \cdot 3 \cdot X^2 
+ \sum_{b\geq 3} 2^{-b-3} \cdot 3 \cdot X^b 
+ 2^{2} \cdot X^{-2}
\\
&  \qquad
+ \sum_{b\geq 1} 2^{2-b} \cdot X^{b-2} 
+ 2^{-6} \cdot 3 \cdot X^2
+ 2^{-9} \cdot 3 \cdot X^4
+ \sum_{2b\geq 6} 2^{-3b-3} \cdot 3 \cdot X^{2b} 
+ 2^{-6}  \cdot X^2
\\
& \qquad
+ \sum_{2b\geq 4} 2^{-3b-3} \cdot X^{2b} 
+ 2^{-7} \cdot 3 \cdot X^3
+ 2^{-8} \cdot 3 \cdot X^4
+ 2^{-10} \cdot 3 \cdot X^5
+ 2^{-11} \cdot 3 \cdot X^6
\\
& \qquad
+ \sum_{c\geq 4} 2^{-c-5} \cdot 3 \cdot X^{c+1}
+ \sum_{2b+1\geq 7} 2^{-3b-4} \cdot 3 \cdot X^{2b+1}  
+ \sum_{2b+2\geq 8} 2^{-3b-5} \cdot 3 \cdot X^{2b+2} 
\\
&  \qquad
+ \sum_{c+2\geq 7} 2^{-c-7} \cdot 3 \cdot X^{c+2} 
+ \sum_{\substack{b+c\geq 9\\ b\geq 3, c \geq b+3}} 2^{-2b-c-3} \cdot 3 \cdot X^{b+c} 
\\
& = 
2^{-1} 
+ 2^{-4} \cdot 3 \cdot X 
+ 2^{-5} \cdot 3 \cdot X^2 
+ \frac{2^{-6} \cdot 3 \cdot X^3}{(1 - \frac{X}{2})}
+ 2^{2} \cdot X^{-2}
\\
&  \qquad
+ \frac{2 \cdot  X^{-1}}{(1 - \frac{X}{2})}
+ 2^{-6} \cdot 3 \cdot X^2
+ 2^{-9} \cdot 3 \cdot X^4
+ \frac{2^{-12} \cdot 3 \cdot X^6}{(1 - \frac{X^2}{8})}
+ 2^{-6}  \cdot X^2
\\
& \qquad
+ \frac{2^{-9} \cdot X^4}{(1 - \frac{X^2}{8})}
+ 2^{-7} \cdot 3 \cdot X^3
+ 2^{-8} \cdot 3 \cdot X^4
+ 2^{-10} \cdot 3 \cdot X^5
+ 2^{-11} \cdot 3 \cdot X^6
\\
& \qquad
+ \frac{2^{-9} \cdot 3 \cdot X^5}{(1 - \frac{X}{2})}
+ \frac{2^{-13} \cdot 3 \cdot X^7}{(1 - \frac{X^2}{8})}
+ \frac{2^{-14} \cdot 3 \cdot X^8}{(1 - \frac{X^2}{8})}
+ \frac{2^{-12} \cdot 3 \cdot X^7}{(1 - \frac{X}{2})}
+ \frac{2^{-15} \cdot 3 \cdot X^9}{(1 - \frac{X}{2})(1 - \frac{X^2}{8})}
\\
& = \frac{(1 - \frac{X^3}{64}) (\frac{X}{2})^{-2}}{(1 - \frac{X}{2})(1 - \frac{X^2}{8})}.
\\
\end{align*}
Now substituting $X = 2^{-s}$ and multiplying by $X^3$ gives the desired Dirichlet series.
\end{proof}


\begin{lem} \label{Lem:B_u_even_prime_formula}
Suppose that $\p \mid 2$, $F_\p = \Q_2$, and we fix some 
$u \in \SqCl(\Op^\times)$.  Then
$$
\sum_{
	\substack{S_\p \in \SqFpInt \\ \text{with $\widetilde{S_\p}= 2u$}}
	}
	\frac{B^*_{\p; n=3}(S_{\p})}{N_{F/\Q}(\I(S_{\p}))^{s}}
=
\frac{-\ve_u \cdot (1 +  2^{-(s+2)} +  2^{-(2s+4)}) \cdot 2^{-s}}
{(1-2^{-(2s+2)})}
$$
as an equality of formal Dirichlet series, where $\ve_u := (-1,u)_\p$.
\end{lem}

\begin{proof}
Since the generic local density at $\p$ is $\beta_{n=3, \p}(2)^{-1} = \frac{4}{1 - \tfrac{1}{2^2}}$,
from Table 5.3 the series  
$4 \cdot \sum_\nu
B^*_{\p}(S_\p)X^{\nu-3}$ with $X := q^{-s}$ and $\nu := \ord_\p(\I(S_\p))$ 
for the Gram determinant (having valuation $\nu-3$)
is given by
\begin{align*}
4 \cdot \sum_{\nu=0}^\infty B^*_{\p}(S_\p)X^{\nu-3}
& = 
-\ve_u \cdot 2^{-2}
+ 0
-\ve_u \cdot 2^{-6} \cdot 3 \cdot X^2
%
+ 
\sum_{b\geq 3, \text{$b$ even}} -\ve_u \cdot 2^{-b-4} \cdot 3 \cdot X^b 
-\ve_u \cdot 2^{2} \cdot X^{-2}
\\
&  \qquad
%
+ \[ 
\sum_{b\geq 1, \text{ even}} -\ve_u \cdot 2^{2-b} \cdot X^{b-2} 
+ \sum_{b\geq 1, \text{odd}} -\ve_u \cdot 2^{1-b} \cdot X^{b-2} 
\]
+ 0
\\
&  \qquad
-\ve_u \cdot 2^{-10} \cdot 3 \cdot X^4
+ \sum_{2b\geq 6, \text{$b$ even}} -\ve_u \cdot 2^{-3b-4} \cdot 3 \cdot X^{2b} 
-\ve_u \cdot 2^{-6}  \cdot X^2
\\
& \qquad
+ \[
\sum_{2b\geq 4, \text{$b$ odd}} -\ve_u \cdot 2^{-3b-3}  \cdot X^{2b} 
+ \sum_{2b\geq 4, \text{$b$ even}} -\ve_u \cdot 2^{-3b-4} \cdot X^{2b} 
\]
+ 0
+ 0
\\
& \qquad
+ 0
-\ve_u \cdot  2^{-12} \cdot 3 \cdot X^6
\\
& \qquad
+ 0
+ 0  
+ \sum_{2b+2\geq 8, \text{$b$ even}} -\ve_u \cdot 2^{-3b-6} \cdot 3 \cdot X^{2b+2} 
\\
&  \qquad
+ \sum_{c+2\geq 7, \text{$c$ even}} -\ve_u \cdot 2^{-c-8} \cdot 3 \cdot X^{c+2} 
+ \sum_{\substack{b+c\geq 9\\ b\geq 3, c \geq b+3 \\ \text{$b$ and $c$ even}}} -\ve_u \cdot 2^{-2b-c-4} \cdot 3 \cdot X^{b+c} 
\\
& = 
-\ve_u \cdot 
\Biggr( \Biggl.
%
2^{-2}
+ 2^{-6} \cdot 3 \cdot X^2
%
+ \frac{2^{-8} \cdot 3 \cdot X^4}{(1 - \frac{X^2}{4})}
+ 2^{2} \cdot X^{-2}
\\
&  \qquad
%
+ \[ 
\frac{1}{(1 - \frac{X^2}{4})}
+\frac{X^{-1}}{(1 - \frac{X^2}{4})}
\]
+  2^{-10} \cdot 3 \cdot X^4
+ \frac{2^{-16} \cdot 3 \cdot X^8}{(1 - \frac{X^4}{64})}
+ 2^{-6}  \cdot X^2
\\
& \qquad
+ \[
\frac{2^{-12} \cdot X^6}{(1 - \frac{X^4}{64})}
+ \frac{2^{-10} \cdot X^4}{(1 - \frac{X^4}{64})}
\]
+  2^{-12} \cdot 3 \cdot X^6
\\
& \qquad
+ \frac{2^{-18} \cdot 3 \cdot X^{10}}{(1 - \frac{X^4}{64})}
+ \frac{2^{-14} \cdot 3 \cdot X^8}{(1 - \frac{X^2}{4})}
+ \frac{2^{-20} \cdot 3 \cdot X^{12}}{(1 - \frac{X^2}{4}) (1 - \frac{X^4}{64})}
\Biggr. \Biggl)
\\
& = \frac{-\ve_u \cdot (1 + \frac{X}{4} + \frac{X^2}{16})}{(1 - \frac{X^2}{4}) \cdot \frac{X^2}{4}}.
\end{align*}
Now substituting $X = 2^{-s}$ and multiplying by $X^3$ gives the desired Dirichlet series.
\end{proof}

\begin{rem}
The dependence on the choice of unit squareclass $u$ 
in Lemma \ref{Lem:B_u_even_prime_formula} as multiplication by $\ve_u$, and the independence of $u$ in Lemma  \ref{Lem:A_u_even_prime_formula}
are general phenomenon that follow for any $\p\mid 2$ (with no condition on $F_\p$) from the local scaling result \cite[Thrm 5.6, p24]{Ha-masses_of_varying_det}) since for $(u,u)_\p = (-1, u)_\p$ for $u \in \SqCl(\Op^\times)$.
\end{rem}


\begin{lem}[Computing Euler Factors]  \label{Lem:Computing_even_euler_factors}
Suppose that $\p\mid 2$ with $F_\p = \Q_2$, 
$\lambda$ is a distinguished family of squareclasses as in \cite[Defn 4.9, pp18-19]{Ha-masses_of_varying_det} with $n=3$, and define
$S_{\p^\nu} := \lambda(\p^\nu)$.
Then the local Euler factors of the Dirichlet series $D_{A^*, \lambda; n=3}(s)$ and $D_{B^*, \lambda; n=3}(s)$ at $\p$ in \cite[Cor 4.13, p20]{Ha-masses_of_varying_det} are explicitly given by 
$$
%
\sum_{\nu \geq 0} 
\frac{A^*_{\p; n=3}(S_{\p^\nu})}{N_{F/\Q}(\I(S_{\p^\nu}))^{s}}
= \frac{(1 - 2^{-(3s+6)}) \cdot 2^{-s}}
{(1-2^{-(s+1)})(1-2^{-(2s+3)})}
$$
and
$$
\sum_{\nu \geq 0} 
\frac{B^*_{\p; n=3}(S_{\p^\nu})}{N_{F/\Q}(\I(S_{\p^\nu}))^{s}}
=
\frac{-\ve_\p \cdot (1 +  \ve_\p \cdot 2^{-(s+2)} +  2^{-(2s+4)}) \cdot 2^{-s}}
	{(1-2^{-(2s+2)})} 
$$
where $\ve_\p := (-1, S_{\p})_\p$.
\end{lem}

\begin{proof}
This follows from Lemmas \ref{Lem:A_u_even_prime_formula} and \ref{Lem:B_u_even_prime_formula} 
%
since the coefficients of our Euler factor Dirichlet series here are taken from the coefficients of the ``fixed $u$'' local Dirichlet series described there.  For the $A^*_\p$-Euler factor there is no dependence of $A^*_\p(S_{\p^\nu})$ on $u$, so the result follows from Lemma \ref{Lem:A_u_even_prime_formula}.  For the $B^*_\p$-Euler factor, the dependence of $B^*_\p(S_{\p^\nu})$ on the unit squareclass alternates the sign of the terms iff the unit squareclass changes residue class (mod $4\Op$) as $\nu$ varies.  When there is no change of sign the result follows from Lemma \ref{Lem:B_u_even_prime_formula}, and when there is a change of sign then the alternating sign of the even $\nu$ terms is achieved by changing the sign of the $2^{-(s+2)}$ term in the numerator of the rational form.  Conveniently, since $\pip = 2$ and $(-1, 2)_\p = 1$ we can write $(-1, u)_\p$ without explicitly normalizing $S_\p$ as $(-1, S_\p)$ if $S_\p = u \cdot \pip^\al$, proving the lemma.
\end{proof}








\section{Computing the Dirichlet series $D_{M^*, \lambda; n=3}(s)$} 

We can now assemble our previous results to give explicit formulas for the Dirichlet series $D_{A^*, \lambda; n=3}(s)$ and $D_{B^*, \lambda; n=3}(s)$ appearing in \cite[Cor 4.13, p20]{Ha-masses_of_varying_det}, which gives our main theorem.  These results allow us to give an explicit formula for the non-archimedean mass Dirichlet series $D_{M^*, \lambda; n=3}(s)$, which is the main object of study in \cite{Ha-masses_of_varying_det}.

\begin{thm}  \label{Thm:Explicit_Dirichlet_Series_Formula}
Suppose that $F$ is a number field in which $p=2$ splits completely.  Then
%
the generic density product  defined in \cite[Defn 4.5, p20]{Ha-masses_of_varying_det} is given by $\beta_{n=3, \mathbf{f}}^{-1}(\widetilde{S}) = 4^{[F:\Q]} \cdot \zeta_F(2)$, and the Dirichlet series 
$$
D_{A^*, \lambda; n=3}(s) = \frac{ \zeta_F(s+1) \zeta_F(2s+3)}{2^{s} \, \zeta_F(3s+6)}.
$$
When all distinguished squareclasses $\lambda(\mathfrak{a})$ are globally rational, then we also have 
$$
D_{B^*, \lambda; n=3}(s)
= 
(-1)^{([F:\Q] + \sigma_{v, -})} \cdot 
\frac{ \zeta_F(2s+2) \zeta_F(s+2)}{2^{s}\, \zeta_F(3s+6)},
$$
where for real places $v$ we have signature $\sigma_v = (\sigma_{v, +}, \sigma_{v, -})$. 
\end{thm}

\begin{proof}
To compute the generic density product, we use our previous results from Case \#1 of Table 2 and Case \#5 of Table 5.3 to see that for $\p\nmid2$ we have the factor $(1-\frac{1}{q^2})^{-1}$ and for each $\p\mid 2$ with $F_\p = \Q_2$ have $4\cdot (1-\frac{1}{q^2})^{-1}$.  Together, these show that $\beta_{n=3, \mathbf{f}}^{-1}(\widetilde{S}) = 4^{[F:\Q]}\zeta(2)$.

The formula for $D_{A^*, \lambda; n=3}(s)$ follows directly from multiplying together the Euler factors computed in Lemmas \ref{Lem:Explicit_Odd_Euler_factors} and \ref{Lem:Computing_even_euler_factors}.  The formula for $D_{B^*, \lambda; n=3}(s)$ follows similarly by using the the identity $(1 \mp X)(1 \pm X + X^2) = (1 \pm X^3)$ to rewrite the the Euler factors for all $\p\nmid 2$ where $\pm := (-1, S)_\p$ and $S:= \lambda(\mathfrak{a})$.  The product of these ``odd'' Euler factors contributing to the $N_{F/\Q}(\mathfrak{a})^{-s}$-term of the Dirichlet series carries an overall sign of $\prod_{p\nmid 2} (-1, S)_\p$, and the Euler factors from the primes $\p\mid 2$ contribute a sign of $\prod_{\p\mid 2} -(-1,S)_\p$.  By the product formula for Hilbert symbols, when $S$ is a globally rational squareclass the product of these signs is the constant $(-1)^{[F:\Q]} \prod_{v\mid \infty_\R} (-1, -1)_v^{\sigma_{v, -}} = (-1)^{[F:\Q] + \sigma_{v, -}}$, proving the formula.
\end{proof}

Specializing the above result to the case $F = \Q$  and the most natural distinguished family of squareclasses gives

\begin{cor}  \label{Cor:Explicit_Dirichlet_Series_Formula_over_QQ}
When $F = \Q$ and we take the distinguished family of squareclasses $\lambda(a) = |a| \in \N$ and consider ternary quadratic forms of signature $\sigma = (\sigma_{+}, \sigma_{-})$, we have
$\beta_{n=3, \mathbf{f}}^{-1}(\widetilde{S}) = 4 \cdot \zeta(2)$, and the Dirichlet series 
$$
D_{A^*, \lambda; n=3}(s) = \frac{ \zeta(s+1) \zeta(2s+3)}{2^{s} \, \zeta(3s+6)}
\qquad\text{ and }\qquad
D_{B^*, \lambda; n=3}(s)
=
(-1)^{(1 + \sigma_{-})}  \cdot 
\frac{ \zeta(2s+2) \zeta(s+2)}{2^{s}\, \zeta(3s+6)}.$$
\end{cor}

\begin{proof}
This follow directly from Theorem \ref{Thm:Explicit_Dirichlet_Series_Formula} by taking $F = \Q$.
\end{proof}


\section{Exact formulas for the formal sum of masses}

In this section we apply our previous computations of the Dirichlet series $D_{A^*, \lambda; n=3}(s)$ and $D_{B^*, \lambda; n=3}(s)$ to give an exact formula for the {\bf primitive total mass Dirichlet series}
$$
D_{\mathrm{Mass}^*, \lambda(\mathfrak{a}), \vec{\sigma}_\infty; n=3}(s)
 := 
\sum_{\mathfrak{a} \in I(\OF)} 
\( 
\sum_{
Q \in \mathbf{Cls}^*(\lambda(\mathfrak{a}), \vec{\sigma}_\infty; n)
}
\frac{1}{|\Aut(Q)|}
\)
N_{F/\Q}(\mathfrak{a})^{-s}
$$
and the {\bf total mass Dirichlet series}
$$
D_{\mathrm{Mass}, \lambda(\mathfrak{a}), \vec{\sigma}_\infty; n=3}(s)
 := 
\sum_{\mathfrak{a} \in I(\OF)} 
\( 
\sum_{
Q \in \mathbf{Cls}(\lambda(\mathfrak{a}), \vec{\sigma}_\infty; n)
}
\frac{1}{|\Aut(Q)|}
\)
N_{F/\Q}(\mathfrak{a})^{-s},
$$
whose coefficients are the total masses of (resp. primitive or all) $\OF$-valued rank 3 quadratic lattices of Hessian determinant squareclass $\lambda(\mathfrak{a})$ having totally definite signature vector $\vec{\sigma}_\infty$.

\begin{thm} \label{Thm:Formal_mass_series_over_F}
Suppose that $p=2$ splits completely in $F$, the signature vector $\vec{\sigma}_\infty$ is totally definite, and the distinguished family of squareclasses $\lambda(\mathfrak{a})$ are all globally rational.  Then 
$$
D_{\mathrm{Mass}^*, \lambda, \vec{\sigma}_\infty; n=3}(s) 
= 
\frac{|\Delta_F|^\frac{3}{2} \cdot \zeta_F(2)}{2 \cdot \(\pi^2 \cdot 2^{s+2} \)^{[F:\Q]}} 
\cdot
\[
\frac{ \zeta_F(s-1) \zeta_F(2s-1)}{ \zeta_F(3s)}
+
\ve \cdot 
\frac{ \zeta_F(2s-2) \zeta_F(s)}{ \zeta_F(3s)}
\]
$$
and
$$
D_{\mathrm{Mass}, \lambda, \vec{\sigma}_\infty; n=3}(s) 
= 
\frac{|\Delta_F|^\frac{3}{2} \cdot \zeta_F(2)}{2 \cdot \(\pi^2 \cdot 2^{s+2} \)^{[F:\Q]}} 
\cdot
\[
\zeta_F(s-1) \zeta_F(2s-1) 
+
\ve \cdot 
\zeta_F(2s-2) \zeta_F(s) 
\],
$$
where $\ve:= (-1)^{[F:\Q]}  \in \{\pm1\}$.
\end{thm}

\begin{proof}
From \cite[Cor 6.4, p26]{Ha-masses_of_varying_det} we have that 
$$
D_{\mathrm{Mass}^*, \lambda, \vec{\sigma}_\infty; n=3}(s) 
=
\frac{|\Delta_F|^\frac{3}{2}}{\( 64\, \pi^2 \)^{[F:\Q]}} 
\cdot
D_{M^*, \lambda, \ve_\infty, \c_\S; n=3}(s - 2),
$$ and also \cite[Cor 4.15, p20]{Ha-masses_of_varying_det} gives
$$
D_{M^*, \lambda, \ve_\infty; n=3}(s) 
= 
\tfrac{1}{2} \beta_{n=3, \mathbf{f}}^{-1}(\widetilde{\lambda(2\OF)}) \cdot 
\[ D_{A^*, \lambda, \ve_\infty; n=3}(s) + \ve_\infty D_{B^*, \lambda, \ve_\infty; n=3}(s)\]
$$
with $\ve_\infty = \prod_{v\mid \infty_\R} c_v(\sigma_v)$.
By computing $c_v(\sigma_v) \cdot (-1)^{\sigma_{v,-}}$ for each of the four real signatures of rank 3,
the result follows from Theorem \ref{Thm:Explicit_Dirichlet_Series_Formula}.

The non-primitive formula follows from the primitive one since a quadratic $\OF$-valued lattice is primitive if it is locally primitive at every prime $\p$.  Summing over all possible scaled locally primitive $\Op$-valued lattices at $\p$ has the effect of multiplying the ``primitive'' Dirichlet series by the local factor $(1 - \frac{1}{N_{\Op/\Q_p}(\I(\pi_p))^{3}})^{-1} = (1 - \frac{1}{q^{3}})^{-1}$, since the scaled lattice $\pip(L_\p, Q_\p) := (L_\p, \pip\cdot Q_\p)$ has Hessian determinant squareclass $\pip^3 \cdot \det_H(L_\p) $.
\end{proof}

In the special case when $F = \Q$ and $\lambda(a\Z) = |a|$, this gives the following useful corollary that gives a good way of numerically verifying the integrity of all of our previous computations.

\begin{cor} \label{Cor:Formal_mass_series_over_QQ}
When $F=\Q$, the formal mass Dirichlet series of the positive definite $\Z$-valued ternary quadratic forms of Hessian determinant  ($\in 2\N$) is given by
$$
D_{\mathrm{Mass}^*; n=3}(s) 
= 
\frac{1}{48 \cdot 2^s} 
\cdot
\[
\frac{ \zeta(s-1) \zeta(2s-1)}{ \zeta(3s)}
-
\frac{ \zeta(2s-2) \zeta(s)}{ \zeta(3s)}
\]
$$
and 
$$
D_{\mathrm{Mass}; n=3}(s) 
= 
\frac{1}{48 \cdot 2^s} 
\cdot
\[
\zeta(s-1) \zeta(2s-1) 
-
\zeta(2s-2) \zeta(s)
\].
$$
\end{cor}

\begin{proof}
This follows directly from Theorem \ref{Thm:Formal_mass_series_over_F} when $F = \Q$, since $|\Delta_F| = 1$ and $\ve = -1$.
\end{proof}

\begin{rem} [Numerical verification]
Corollary \ref{Cor:Formal_mass_series_over_QQ} provides an independent global method for checking the accuracy of our previous computations of the Dirichlet series $D_{A^*; n=3}(s)$ and $D_{B^*; n=3}(s)$ by enumerating the reduced classes of ternary quadratic forms $Q$ of Hessian determinant $\det_H(Q) \leq X$, and then comparing the results.  The exact inequalities defining reduced positive definite ternary quadratic forms are worked out in 
\cite[]{Dickson:1957nx}, and we have used them to check Corollary \ref{Cor:Formal_mass_series_over_QQ} when $\det_H(Q) \leq 20,000$.  
To perform  these computations we have written specialized software that allows formal Dirichlet series manipulations \cite{Ha-Sage-Dirichlet}, creates and queries tables of ternary quadratic forms \cite{Ha-Sage-Tables}, and implements the local computations of the terms in $A^*_{p; n=3}(S), B^*_{p; n=3}(S)$ and the mass Dirichlet series \cite{Ha-Sage-Mass}.
These are written in Python 2.6 \cite{Python} for the open-source and freely available Sage computer algebra system \cite{SAGE}.  They are built on the functionality of the {\tt QuadraticForm()} class developed in \cite{Ha-Sage-QF-class} and may be downloaded from \cite{hanke_webpage}.
\end{rem}

\bibliographystyle{plain}	
\bibliography{myrefs}		

\end{document}